\documentclass[opre,nonblindrev]{informs3} 

\usepackage{amssymb}
\usepackage{color}
\usepackage{rotating}
\usepackage{multirow}
\usepackage{latexsym}
\usepackage{secdot}
\usepackage{setspace}
\usepackage{wasysym}
\usepackage[us,12hr]{datetime}
\usepackage{bm}
\usepackage{pdfsync}
\usepackage{rccol}
\usepackage{mathrsfs}
\usepackage{epsfig}
\usepackage[margin=1in,paper=letterpaper]{geometry}
\usepackage{url}
\usepackage{multibib}
\newcites{appendix}{Online Appendix References}
\usepackage{enumitem}
\usepackage{stackrel}
\usepackage{soul}
\usepackage{appendix}
\usepackage{titlesec}
\usepackage{booktabs}
\usepackage{longtable}
\usepackage{float}
\usepackage{placeins}

\newtheorem{thm}{Theorem}[section]
\newtheorem{corollary}[thm]{Corollary}
\newtheorem{lem}[thm]{Lemma}
\newtheorem{prop}[thm]{Proposition}

{\theoremstyle{THkey}}

\OneAndAHalfSpacedXI 

\usepackage{natbib}
 \bibpunct[, ]{(}{)}{,}{a}{}{,}%
 \def\bibfont{\small}%
 \def\bibsep{\smallskipamount}%

\EquationsNumberedThrough    

\MANUSCRIPTNO{} 

\newcommand{\ts}{{\,}}

\newcommand{\Ccal}{{\mathcal C}}

\newcommand{\Jcal}{{\mathcal J}}

\newcommand{\Ncal}{{\mathcal N}}
\newcommand{\Pcal}{{\mathcal P}}

\newcommand{\Tcal}{{\mathcal T}}
\newcommand{\Xcal}{{\mathcal X}}

\newcommand{{\qhat}}{{\widehat q}}
\newcommand{{\xhat}}{{\widehat x}}
\newcommand{{\yhat}}{{\widehat y}}
\newcommand{{\zhat}}{{\widehat z}}
\newcommand{{\alphahat}}{{\widehat \alpha}}

\newcommand{{\pbar}}{{\overline p}}
\newcommand{{\ubar}}{{\overline u}}
\newcommand{{\xbar}}{{\overline x}}
\newcommand{{\ybar}}{{\overline y}}
\newcommand{{\zbar}}{{\overline z}}
\newcommand{{\Jbar}}{{\overline J}}
\newcommand{{\Cbar}}{{\overline C}}
\newcommand{{\Pbar}}{{\overline P}}
\newcommand{{\Rbar}}{{\overline R}}
\newcommand{{\betabar}}{{\overline \beta}}
\newcommand{{\nubar}}{{\overline \nu}}
\newcommand{{\thetabar}}{{\overline \theta}}
\newcommand{{\zetabar}}{{\overline \zeta}}
\newcommand{{\Deltabar}}{{\overline \Delta}}
\newcommand{{\Brm}}{\text{\rm B}}
\newcommand{{\Grm}}{\text{\rm G}}
\newcommand{{\Mrm}}{\text{\rm M}}
\newcommand{{\Nrm}}{\text{\rm N}}
\newcommand{{\Rrm}}{\text{\rm R}}
\newcommand{{\Srm}}{\text{\rm S}}
\newcommand{{\Urm}}{\text{\rm U}}
\newcommand{{\Xrm}}{\text{\rm X}}
\newcommand{{\Yrm}}{\text{\rm Y}}
\newcommand{{\Fs}}{{\text{\sf F}}}
\newcommand{{\Ss}}{{\text{\sf S}}}
\newcommand{{\sigmahat}}{{\widehat \sigma}}

\newcommand{\evec}{{\bm e}}
\newcommand{\pvec}{{\bm p}}
\newcommand{\qvec}{{\bm q}}
\newcommand{\uvec}{{\bm u}}

\newcommand{\xvec}{{\bm x}}
\newcommand{\yvec}{{\bm y}}
\newcommand{\wvec}{{\bm w}}

\newcommand{\alphavec}{{\bm \alpha}}
\newcommand{\etavec}{{\bm \eta}}

\newcommand{\phivec}{{\bm \varphi}}
\newcommand{\pvechat}{\widehat{\bm p}}
\newcommand{\qvechat}{\widehat{\bm q}}

\newcommand{\alphavechat}{\widehat{\bm \alpha}}

\newcommand{{\Prmbar}}{{\overline{\text{P}}}}

\newcommand{\opt}{{\text{\sf opt}}}
\newcommand{\rev}{{\text{\sf rev}}}

\renewcommand{\qed}{\hfill \mbox{\raggedright \rule{0.1in}{0.1in}}}
\newcommand{\ind}[1]{{\bf 1}{(#1)}}

\begin{document}

\RUNAUTHOR{}
\RUNTITLE{\normalfont{De-Randomization  for Dynamic Assortment Optimization}}

\TITLE{\large 
Killing the Case for Randomization in Dynamic Assortment Optimization}

\ARTICLEAUTHORS{
\vspace{-3mm}
\AUTHOR{\mbox{Mikhail Fadin, Huseyin Topaloglu}}
\AFF{
School of Operations Research and Information Engineering, Cornell Tech, New York, NY 10044, USA}
\vspace{-1.5mm}
\EMAIL{\footnotesize mf853@cornell.edu,~topaloglu@orie.cornell.edu}
\\
\vspace{-0mm}
{\normalsize June 24, 2026}
\vspace{-3mm}
}

\ABSTRACT{%
One of the traditional approaches for constructing approximate policies for dynamic assortment optimization problems is to use sampling-based inventory-agnostic policies. Such policies are called sampling-based, as they sample an assortment of products from a fixed distribution at each time period to offer to a customer of each type. Such policies are called inventory-agnostic, as the sampled assortments may include products without remaining inventories, so if a customer chooses a product without remaining inventories, then she leaves without a purchase. Inventory-agnostic nature of a policy is not a concern, because it is known that if the policy samples an assortment that includes products without remaining inventories, then dropping the products without remaining inventories does not degrade the performance. However, sampling-based nature of a policy is a concern, because sampling brings another source of uncertainty in the performance and may erode trust in the platform. In this paper, we give an algorithm to de-randomize any \mbox{sampling-based} inventory-agnostic policy, so the de-randomized policy offers a deterministic sequence of assortments within the support of the original policy without degrading the performance. 
Furthermore, we give a variation of our de-randomization algorithm that searches for a deterministic sequence of assortments beyond the support of the original policy.~We~show~that we can implement the latter variation efficiently as long as we can solve the static assortment optimization problem under the choice model governing the choice process of the customers. 
Either way, if the original policy has a performance guarantee, then the de-randomized policy performs at least just as well. As our crowning technical contribution, we study locally-optimal deterministic policies, where changing any single one of the assortments in the policy does not improve the total expected revenue. We show that \underline{\it any} locally-optimal policy has a performance guarantee of $\frac 12 -\epsilon$ when compared with the best sampling-based policy. Noting that the best sampling-based policy has a \mbox{constant factor} performance guarantee, our result yields a fully novel approach for dynamic assortment optimization. Our computational work indicates that a de-randomized policy can substantially improve the original policy and working with assortments beyond the support of the original policy can be quite useful.
}



\maketitle 

\vspace{-7mm}

\section{Introduction}

\vspace{-1mm}

In dynamic assortment optimization problems, we have products with limited inventories. During the course of the selling horizon customers with different types arrive into the system. Observing the type of the arriving customer, we offer an assortment of products as a function of the remaining inventories of the products and time left in the selling horizon. The customer either makes a purchase within the offered assortment or decides to leave without making a purchase. The goal is to find a policy to pick the assortment of products to offer to each customer to maximize the total expected revenue over the selling horizon. Dynamic assortment optimization problems started gaining attention over the past decade, partly due to their ability to help online retailers customize the assortments of products offered to their customers as a function of the information about the preferences of each customer and remaining inventories of the products. Dynamic programming formulation  of the dynamic assortment optimization problem involves a high-dimensional state variable, keeping track of the remaining inventories of the products. Thus, it is computationally difficult to compute the optimal policy. One of the traditional approaches for constructing approximate policies with performance guarantees is to use sampling-based inventory-agnostic policies. Such policies are sampling-based, as they sample an assortment of products from a fixed distribution that depends on the type of the arriving customer and time left in the selling horizon.  Such policies are  inventory-agnostic, as the sampled assortments may include products without remaining inventories, so if a customer chooses a product without remaining inventories, then she leaves without a purchase. There exist variety of sampling-based inventory-agnostic policies with performance guarantees. Inventory-agnostic nature of a policy is not a concern because it is known that if we drop the products without remaining inventories from the sampled assortment, then the performance of a policy does not degrade, as long as the choice model for the customers satisfies mild conditions. However, the sampling-based nature of a policy is a concern, because sampled assortments bring an artificial source of uncertainty in the performance.

We focus on de-randomizing a sampling-based inventory-agnostic policy. Given a \mbox{sampling-based} inventory-agnostic policy, we give an algorithm to come up with a policy that offers a deterministic sequence of assortments within the support of the random assortments of the original policy, where each element in the sequence corresponds to the fixed assortment that we offer to a customer of a particular type at a particular time period. While the \mbox{de-randomized} policy is only guaranteed to not degrade the performance of the original policy, it often provides significant improvements in practice. We also show that if we can solve the static assortment optimization problem under the choice model governing the choices of the customers, then we can come up with another \mbox{de-randomized} policy that is not limited to the original support and can perform better. Thus, by looking at assortments beyond the support of the original policy, we can potentially obtain even better de-randomized policies. Lastly, we consider \mbox{locally-optimal} deterministic policies, where changing any single one of the assortments in the policy does not improve its performance. We show that such policies have a performance guarantee of \mbox{$\frac 12 - \epsilon$} when compared with the best sampling-based policy. We give an efficient algorithm to find a \mbox{locally-optimal} policy, resulting in a novel approach for dynamic assortment optimization.

{\bf \underline{Technical Contributions\phantom{p}\!\!\!}:} We give algorithms to de-randomize any policy within and beyond the support of its assortments. We provide a performance guarantee for locally-optimal policies.

{\it \underline{De-Randomizing a Sampling-Based Inventory-Agnostic Policy}.} Starting from a sampling-based inventory-agnostic policy, we give an algorithm to de-randomize the policy in running time polynomial in  input size, so that we get a policy that offers a deterministic sequence of assortments. There are two pieces to our algorithm. First, we, naturally, show that if we replace one of the random assortments in a sampling-based inventory-agnostic policy with the best possible deterministic assortment while keeping the other parts of the policy untouched, then the total expected revenue of the policy does not degrade. Second, we give an efficient approach to compute the total expected revenue of a sampling-based inventory-agnostic policy. To establish both of the pieces, it is crucial that the policy on hand is both sampling-based and inventory-agnostic, so we strictly exploit the fact that we work with sampling-based inventory-agnostic policies. Letting $T$ be the number of time periods in the selling horizon, $n$ be the number of products and $m$ be the number of customer types, our algorithm de-randomizes a policy in a running time of $O((m n+T)\ts nT^2)$. In this running time, we use the Caratheodory theorem to argue that a randomized policy never needs to randomize over more than $n+1$ assortments. Letting $c_{\min}$ be the smallest initial inventory of a product, there are sampling-based inventory-agnostic policies that attain constant factor performance guarantees or asymptotic optimality at rate $1- O\Big(\frac{1}{\sqrt{c_{\min}}}\Big)$. De-randomizing these policies using our work, we get the same performance guarantees using a deterministic sequence of assortments. It is remarkable that we can attain asymptotic optimality through a deterministic sequence of assortments. 

{\it \underline{Beyond the Support of a Sampling-Based Inventory-Agnostic Policy}.} Once we de-randomize a policy, we obtain a deterministic sequence of assortments, each assortment within the support of the original policy, such that the de-randomized policy is guaranteed to perform at least as well as the original policy, as discussed in the previous paragraph. We may obtain even better \mbox{de-randomized} policies by considering assortments outside the support of the original policy, but it is not clear that we can efficiently find the best possible assortment outside the support of the original policy. Assuming that we can solve the static assortment optimization problem under the choice model governing the choices of the customers, we give an algorithm that picks the best possible deterministic assortment when de-randomizing a policy, even when the best possible assortment is chosen outside the support of the original policy. The result is \mbox{non-trivial} because if we change the assortment offered at a particular time period, then we change the probability of running out of the inventory of a product at that time period, which has implications on the expected revenue at other time periods. We exploit a total expected revenue expression for a sampling-based inventory-agnostic policy. Static assortment optimization problem under many choice models, including the multinomial logit, nested logit and Markov chain choice models, admits an efficient solution, so our assumption is mild.

{\it \underline{Half-Approximation for Locally-Optimal Policies}.} While they have substantial implications and have never been used to de-randomize a policy, some of the techniques that we use for de-randomization are considered folklore. However, as our crowning technical contribution, we give a performance guarantee for locally-optimal policies. In a locally-optimal policy, we offer a deterministic sequence of assortments and changing any single assortment does not improve the performance of the policy. We say that the policy is approximately locally-optimal if changing any of the deterministic assortments does not improve the performance by more than $\epsilon$. We show that an approximately locally-optimal policy has a performance guarantee of $\frac 12 - O(\epsilon)$ when compared with the best sampling-based policy. We show that we can find an approximately locally-optimal policy in running time polynomial in input size and $\frac 1 \epsilon$. This result provides an entirely novel approach for computing policies with performance guarantees. In particular, because we can efficiently compute the total expected revenue of any inventory-agnostic policy, we can get a policy with a performance guarantee by exchanging the assortments offered by the policy with the best alternatives and checking whether the exchange provides an improvement in the performance.

Our performance guarantee for approximately locally-optimal policies has multiple novel steps using recent results on continuous DR-submodular functions. We argue that searching over inventory-agnostic policies is equivalent to searching over the purchase probability of each product at each time period. Following this result, we show that the total expected revenue of an inventory-agnostic policy is a monotone and continuous DR-submodular function of the purchase probabilities. Considering the problem of maximizing a monotone and continuous DR-submodular function, using the results in \cite{HaSo17}, we show that a stationary point yields a  half-approximation. Finally, we show that the product purchase probabilities of a locally-optimal policy yield a stationary point of the total expected revenue when viewed as a function of the purchase probabilities of the products. Collecting these results yields the performance guarantee. To our knowledge, DR-submodularity had no prior use for optimizing the parameters of a policy.

{\it \underline{Inventory-Aware Counterpart of a Policy}.} Whether we de-randomize a policy or use a locally-optimal policy, we end up with a policy that offers a deterministic sequence of assortments, but the policy is still inventory-agnostic in the sense that it can offer an assortment that includes products without remaining inventories. If a customer chooses a product without remaining inventories, then she leaves without a purchase. Offering products without remaining inventories is not sensible in many practical settings. To complete the discussion, we show that if a policy attempts to offer an assortment that includes products without remaining inventories, then dropping the products without remaining inventories never degrades the performance of the policy, as long as the choice model governing the choices of the customers satisfies the substitutability assumption. Thus, the fact that our policies are inventory-agnostic is not a concern, because we can implement their inventory-aware counterpart without degrading their performance. 

\vspace{-1mm}

{\bf \underline{Related Literature\phantom{p}\!\!\!}:} There is a growing body of literature on dynamic assortment optimization problems.  \cite{MaSi21} give policies with constant factor guarantees, as well as asymptotic optimality guarantees as the initial inventories of the products get large. The authors visit de-randomization, but their de-randomization approach uses simulation to estimate certain total expected revenues, so the policies are not fully de-randomized as they are subject to simulation error. \cite{FeNi24} give policies with asymptotic optimality guarantees. To ensure that their policies do not offer products without remaining inventories, they give a procedure that either removes the products without remaining inventories or offers the empty assortment without changing the purchase probability of each product. The possibility of offering the empty assortment is useful to establish a performance guarantee, but can degrade the practical performance of the policy. \cite{BaEl25} give a single policy with both constant factor and asymptotic optimality guarantees. To our knowledge, theirs is the first work to show that dropping the products without remaining inventories does not degrade the performance of a policy. We include an analogous result for completeness. \cite{SuUd25} focus on the problem when the customers choose according to the multinomial logit model and give improved constant factor guarantees, while also allowing the number of customer arrivals over the selling horizon to be a general random variable.

The papers discussed so far use an optimal solution to a linear program to guide the decisions of the policy. Our de-randomization approach can work with any sampling-based inventory-agnostic policy that may or may not have been obtained through a linear program. Also, our performance guarantee for the locally-optimal policies is very different, because a locally-optimal policy can be obtained through coordinate-wise exchanges of the assortments starting from any policy. 
On the other hand, \cite{GoNaRu14} give performance guarantees by scaling the revenue associated with a product by a factor that depends on its remaining inventory and following a myopic policy. \cite{RuSuTo17} give a constant factor performance guarantee by approximating the value functions. \cite{MaRuSuTo18} extend this work to a network of resources by using more involved value function approximations. \cite{GoGo22} give a competitive ratio for myopic policies. \cite{MaRo22} use an optimal solution to a linear program to guide the evaluation of the value of a unit of product. Work on dynamic assortment optimization goes back to network revenue management with customer choice; see, for example, \cite{GaIy04}, \cite{RyLi04}, \cite{JaKu12}, \cite{To13},  \cite{VoZh13}. 

{\bf \underline{Organization}:} In Section \ref{sec:form}, we describe the problem setting and sampling-based inventory-agnostic policies. In Section \ref{sec:derand}, we give our approach for de-randomizing a policy by considering assortments within and beyond the support of the random assortments of the original policy. In Section \ref{sec:local_opt}, we give a performance guarantee for an approximately locally-optimal policy, along with an algorithm to compute one. In Section \ref{sec:inv_aware}, we show that dropping products without remaining inventories from the assortments offered by a policy does not degrade the performance of the policy. In Section \ref{sec:numerics}, we give our numerical experiments.




%

\newpage

\section{Model}
\label{sec:form}

\vspace{-3mm}

We have $n$ products indexed by $\Ncal = \{1,\ldots,n\}$. The revenue of product $i$ is $r_i \geq 0$. The initial inventory of product $i$ is $c_i$. We have $m$ customer types indexed by $\Jcal = \{1,\ldots,m\}$. We divide the selling horizon into a number of time periods, each corresponding to a small enough interval of time that there is one customer arrival at each time period. We have $T$ time periods in the selling horizon indexed by $\Tcal = \{1,\ldots,T\}$. A customer of type $j$ arrives at time period $t$ with probability $\lambda_{jt}$. We have one customer arrival at each time period, so $\sum_{j \in \Jcal} \lambda_{jt} = 1$.~If there is a positive probability of no customer arrival, then we can define a dummy customer type that does not choose any of the products and set the arrival probability of the dummy customer type as the no customer arrival probability. If we offer the assortment of products $S \subseteq \Ncal$ to a customer of type~$j$, then the customer purchases product $i$ with probability $\phi_{ij}(S)$. We have $\phi_{ij}(S) = 0$ for all $i \not \in S$. The choices of the customers are governed by a choice model with the substitutability property, which is to say that the choice probability of a product cannot increase when we add a product to the assortment. In other words, we have $\phi_{ij}(S \cup \{\ell\}) \leq \phi_{ij}(S)$ for all $i \in S$, $\ell \not \in S$, $j \in \Jcal$ and $S \subseteq \Ncal$. A choice model satisfies the substitutability property as long as it is based on random utility maximization, where a customer associates random utilities with the options and chooses the option with the largest utility. A broad class of choice models are based on random utility maximization.

We focus on \underline{\it sampling-based inventory-agnostic} policies. Our policies are sampling-based in the sense that they sample an assortment of products from a fixed distribution at each time period to offer to a customer of each type without taking into account the remaining inventories of the products. Moreover, our policies are inventory-agnostic in the sense that the assortment they sample may include products without remaining inventories. If the customer chooses a product without remaining inventory, then the customer leaves without a purchase. Throughout the paper, unless we state otherwise, we mean a sampling-based inventory-agnostic policy when we refer to a policy. We use the random assortment $\Srm_{jt}^\mu$ to capture the assortment that policy $\mu$ offers to a customer type $j$ at time period~$t$.~The random assortments $(\Srm_{jt}^\mu : j \in \Jcal,~t \in \Tcal)$ are all independent.~Using inventory-agnostic policies is not a practical concern, because if the customers choose according to a choice model that satisfies the substitutability property, then we can ensure that an inventory-agnostic policy never offers a product without remaining inventories, while also ensuring that the total expected revenue of the policy does not degrade. In particular, we can show that if we drop the products without remaining inventories from the assortment sampled by an inventory-agnostic policy, then the total expected revenue of the policy does not degrade. We will re-visit ensuring that an inventory-agnostic policy never offers a product without remaining inventories.

Considering policy $\mu$ characterized by the random assortments $(\Srm_{jt}^\mu : j \in \Jcal,~t \in \Tcal)$, we use $\Omega_{jt}^\mu$ to denote the support of the random assortment $\Srm_{jt}^\mu$. In this case, if a customer of type $j$ arrives at time period $t$, then policy $\mu$ offers assortment $S \in \Omega_{jt}^\mu$ to the customer with probability $\mathbb P \{ \Srm_{jt}^\mu = S\}$. If the customer chooses a product without remaining inventories, then the customer leaves without a purchase. Otherwise, the inventory of the chosen product is depleted by one unit, accumulating the revenue corresponding to the chosen product. 
There are a number of \mbox{sampling-based} \mbox{inventory-agnostic} policies that provide performance guarantees when compared with the optimal policy that offers assortments by taking into account the remaining capacities of the resources and never offering any products without remaining inventories. To our knowledge, all of these policies are obtained by solving a linear programming approximation and offering assortments sampled according to an optimal solution to the linear programming approximation; see \cite{MaSi21}, \cite{FeNi24} and \cite{BaEl25}. The sampling-based nature of these policies is unappealing and impractical, as randomly sampled assortments bring an artificial source of uncertainty in the performance of the policy and there already exists an optimal policy with decisions deterministically depending on the remaining inventories of the products and time left in the selling horizon. Given a sampling-based inventory-agnostic policy, we study constructing policies that offer a deterministic sequence of assortments without degrading the performance, as well as constructing locally-optimal policies with performance guarantees. 

\vspace{-1mm}

\section{De-Randomizing an Inventory-Agnostic Sampling-Based Policy}
\label{sec:derand}

\vspace{-1mm}

We focus on de-randomizing a sampling-based inventory-agnostic policy without degrading the total expected revenue of the policy. We follow two steps. First, we show that if we replace one of the random assortments in a  policy with the best possible deterministic assortment keeping the other parts of the policy untouched, then the total expected revenue of the policy does not decrease. By the best possible deterministic assortment, we mean the one that maximizes the total expected revenue of the policy when we keep the other parts of the policy untouched. Second, we show that we can find the best possible deterministic assortment efficiently. Considering policy $\mu$ characterized by the random assortments $(\Srm_{jt}^\mu : j \in \Jcal,~t \in \Tcal)$, recall that we use $\Omega_{jt}^\mu$ to denote the support of the random assortment $\Srm_{jt}^\mu$.~We use $\Mrm_{jt}(\mu,S)$ to denote the policy obtained by replacing the random assortment $\Srm_{jt}^\mu$ in policy $\mu$ with the deterministic assortment $S$. In the next proposition, we show that the total expected revenue of policy $\Mrm_{jt}(\mu,S)$ for some assortment $S \in \Omega_{jt}^\mu$ is at least as large as that of policy $\mu$. The intuition behind this result is that if we condition on having $\Srm_{jt}^\mu = S$, then the total expected revenue of policy $\mu$ is precisely equal to the total expected revenue of policy $\Mrm_{jt}(\mu,S)$. Thus, we can compute the total expected revenue of policy $\mu$ by conditioning on having $\Srm_{jt}^\mu = S$ for all $S \in \Omega_{jt}^\mu$ and computing the total expected revenue of policy $\Mrm_{jt}(\mu,S)$. Throughout the paper, we use $\rev(\mu)$ to denote the total expected revenue of policy $\mu$.

\begin{prop}
[Replacing Random Assortments]
\label{pro:derand}
Using $\rev(\mu)$ to denote the total expected revenue of policy $\mu$, we have 
\begin{align*}
\rev(\mu) 
~\leq~
\max_{S \in \Omega_{jt}^\mu} \rev(\Mrm_{jt}(\mu,S)).
\end{align*}

\end{prop}

\vspace{-3mm}

\noindent{\it Proof.} Because the assortments that policy $\mu$ samples for different customer types and at different time periods are independent of each other, as well as independent of the arrivals and choices of the customers, conditional on having $\Srm_{jt}^\mu = S$, the total expected revenue of policy $\mu$ is the total expected revenue of policy $\Mrm_{jt}(\mu,S)$. Thus, we have $\rev(\mu) = \sum_{S \in \Omega_{jt}^\mu} \mathbb P \{ \Srm_{jt}^\mu = S \} \ts \rev(\Mrm_{jt}(\mu,S))$, so there must exist an assortment $S \in \Omega_{jt}^\mu$ such that $\rev(\Mrm_{jt}(\mu,S)) \geq \rev(\mu)$.   \qed

By the proposition above, there exists a deterministic assortment in the support $\Omega_{jt}^\mu$ such that if we replace the random assortment $\Srm_{jt}^\mu$ with this deterministic assortment, then we obtain a policy that performs at least as well as policy $\mu$. We focus on the maximization problem in the proposition. Under policy $\mu$, the demand for product $i$ from a customer of type $j$ at time period~$t$ is a Bernoulli random variable with parameter $\eta_{ijt}^\mu = \sum_{S \in \Omega_{jt}^\mu} \mathbb P \{ \Srm_{jt}^\mu = S \} \ts \phi_{ij}(S)$. By the last equality, using the $n$-dimensional vectors $\etavec_{jt}^\mu = (\eta_{ijt}^\mu : i \in \Ncal)$ and $\Phi_{jt}(S) = (\phi_{ij}(S) : i \in \Ncal)$, we can represent the vector $\etavec_{jt}^\mu$ by using a convex combination of the vectors $\{ \Phi_{jt}(S) : S \in \Omega_{jt}^\mu\}$. In this case, by the Caratheodory theorem, it follows that we can represent the vector~$\etavec_{jt}^\mu$ by using a convex combination of at most $1+n$ of the vectors $\{ \Phi_{jt}(S) : S \in \Omega_{jt}^\mu\}$, so we can proceed with the understanding that the support $\Omega_{jt}^\mu$ includes at most $1+n$ assortments. Therefore, we can solve the maximization problem in the proposition simply by enumeration over the support $\Omega_{jt}^\mu$.  

To solve the maximization problem in the proposition through enumeration, we need to be able to compute the total expected revenue $\rev(\Mrm_{jt}(\mu,S))$, which is our next focus.

\vspace{0.5mm}

{\bf \underline{Computing the Total Expected Revenue of a Policy}:}
\\
\indent We focus on computing the total expected revenue of a sampling-based inventory-agnostic policy. Under policy $\mu$, we offer the assortment $S$ to a customer of type $j$ at time period~$t$ with probability $\mathbb P \{ \Srm_{jt}^\mu = S \}$. Thus, under policy $\mu$, the demand for product $i$ at time period $t$ is a Bernoulli random variable with parameter $\varphi_{it}^\mu = \sum_{j \in \Jcal}  \lambda_{jt} \sum_{S \in \Omega_{jt}^\mu} \mathbb P \{ \Srm_{jt}^\mu = S \} \ts \phi_{ij}(S)$. In this case, using $\Brm_{it}^\mu$ to denote a Bernoulli random variable with parameter $\varphi_{it}^\mu$, the total expected sales for product $i$ under policy $\mu$ is \mbox{$G_i(\mu) = \mathbb E \{ \min\{ c_i , \sum_{t \in \Tcal} \Brm_{it}^\mu\}\}$}, in which case, the total expected revenue of policy $\mu$ is given by $\rev(\mu) = \sum_{i \in \Ncal} r_i \ts G_i(\mu)$. Because the assortments offered by policy $\mu$ at different time periods are independent, the vectors of random variables $(\Brm_{it}^\mu : i \in \Ncal)$ and $(\Brm_{is}^\mu : i \in \Ncal)$ are independent of each other for $t \neq s$. The random variable $\sum_{t \in \Tcal} \Brm_{it}^\mu$ involves a sum of independent Bernoulli random variables, so characterizing the distribution of this random variable is difficult, but we can use a dynamic program to compute the total expected sales $G_i(\mu)$ for product $i$. The state variable in the dynamic program is the remaining inventory of product $i$. Using the boundary conditions $J_{it}(0;\mu) = 0$ for all $t \in \Tcal$ and $J_{i,T+1}(q;\mu) = 0$ for all $q=0,\ldots,c_i$, consider the dynamic program
\begin{align}
J_{it}(q;\mu) = \varphi_{it}^\mu \ts \Big\{ 1 + J_{i,t+1}(q-1; \mu) \Big\} + (1 - \varphi_{it}^\mu) \ts J_{i,t+1}(q; \mu). 
\label{eqn:sales_dp}
\end{align}
In (\ref{eqn:sales_dp}), we have a sale for product $i$ at time period $t$ with probability $\varphi_{it}^\mu$, in which case, we accumulate one unit of sale and deplete one unit of inventory. We do not have a sale for product $i$ at time period $t$ with probability $1- \varphi_{it}^\mu$, in which case, we do not accumulate any sales or deplete any inventory. By the boundary condition $J_{it}(0;\mu)= 0$, if we run out of inventory for product $i$, then we do not accumulate any sales. Computing the value functions $\{ J_{it}(\cdot;\mu) : t \in \Tcal\}$ through the dynamic program above, we have $G_i(\mu) = J_{i1}(c_i;\mu)$. Noting that $|\Omega_{jt}^\mu| = O(n)$, we can \mbox{pre-compute}  $\varphi_{it}^\mu = \sum_{j \in \Jcal}  \lambda_{jt} \sum_{S \in \Omega_{jt}^\mu} \mathbb P \{ \Srm_{jt}^\mu = S \} \ts \phi_{ij}(S)$ for all $i \in \Ncal$ and $t \in \Tcal$ in $O(mn^2T)$ operations. Thus, we can compute $J_{it}(q;\mu)$ for one pair $(q,t) \in \mathbb Z \times \Tcal$ in $O(1)$ operations when solving the dynamic program in (\ref{eqn:sales_dp}). Because there is at most one unit of demand at each time period, we can assume that the initial inventory of each product is at most $T$, so we can compute  $J_{it}(q;\mu)$ for all $q =0,\ldots,c_i$, $i \in \Ncal$ and $t \in \Tcal$ in $O(nT^2)$ operations. Therefore, we can solve the dynamic program in (\ref{eqn:sales_dp}) for all $i \in \Ncal$ in $O((m n+T)\ts nT)$ operations, yielding the total expected revenue of policy $\mu$.

{\bf \underline{De-Randomizing a Policy}:}
\\
\indent By Proposition \ref{pro:derand}, if we replace one of the random assortments in a policy with the best possible deterministic assortment in its support, then the total expected revenue of the policy does not degrade. Furthermore, once we replace one of the random assortments in a policy with a deterministic assortment, we can efficiently compute the total expected revenue of the resulting policy by using the dynamic program in (\ref{eqn:sales_dp}). Using these two results, we give an algorithm to de-randomize a \mbox{sampling-based} inventory-agnostic policy $\mu$. In our algorithm, we loop over all time periods and customer types. For each time period~$t$ and customer type $j$, we replace the random assortment $\Srm_{jt}^\mu$ with the best possible deterministic assortment in the support of $\Srm_{jt}^\mu$. To find the best possible deterministic assortment, we check the total expected revenue of policy $\Mrm_{jt}(\mu,S)$ for all $S \in \Omega_{jt}^\mu$ and pick the assortment that provides the largest total expected revenue. Once we loop over all time periods and customer types, we have a policy that offers a deterministic sequence of assortments to all customer types at all time periods. We consider the following algorithm. 

\noindent {\bf Step 0.} Start with a sampling-based inventory-agnostic policy $\mu$ to de-randomize.
\\
\noindent {\bf Step 1.} For all $t =1,\ldots,T$ and $j=1,\ldots,m$ do;
\\
\indent ~~~~~~~~\ts Setting $S^* = \arg\max_{S \in \Omega_{jt}^\mu} \rev(\Mrm_{jt}(\mu,S))$, redefine policy $\mu$ as $\Mrm_{jt}(\mu,S^*)$.

We give an alternative characterization of the total expected revenue of a policy that becomes useful at numerous points in the paper. One immediate use is that it avoids computing $\rev(\Mrm_{jt}(\mu,S))$ from scratch for each $S \in \Omega_{jt}^\mu$ in our algorithm. Letting $\ind{\cdot}$ be the indicator function, by the identity $\min\{ c_i , \sum_{k \in \Tcal} \Brm_{ik}^\mu\} = \min\{ c_i , \sum_{k \in \Tcal \setminus \{t\}} \Brm_{ik}^\mu\} + \Brm_{it}^\mu \ts \ind{\sum_{k \in \Tcal \setminus \{t\}} \Brm_{ik}^\mu \leq c_i - 1} $, we have
%
%
\begin{align}
G_i(\mu) = \mathbb E\Bigg\{ \min \Big\{ c_i ,\!\!\! \sum_{k \in \Tcal \setminus \{t\}} \Brm_{ik}^\mu \Big\}\Bigg\} + \varphi_{it}^\mu \ts\ts \mathbb P \Bigg\{\sum_{k \in \Tcal \setminus \{t\}} \Brm_{ik}^\mu \leq c_i - 1 \Bigg\}.
\label{eqn:rev_dec}
\end{align}
%
%
\noindent Consider the loop in our algorithm for different customer types for \underline{\it fixed} time period~$t$. Policy $\Mrm_{jt}(\mu,S)$ offers the same assortment as policy $\mu$ at all time periods other than time period $t$, so the expectation and probability in (\ref{eqn:rev_dec}) for policy $\Mrm_{jt}(\mu,S)$ are the same as those for policy $\mu$. Thus, considering the loop in our algorithm for fixed time period $t$, we compute the expectation \mbox{$G_{it}^-(\mu) =  \mathbb E\{ \min \{ c_i , \sum_{k \in \Tcal \setminus \{t\}} \Brm_{ik}^\mu \}\}$} and probability $H_{it}^-(\mu) = \mathbb P \{\sum_{k \in \Tcal \setminus \{t\}} \Brm_{ik}^\mu \leq c_i - 1 \}$ once for all $i \in \Ncal$. We use the dynamic program in (\ref{eqn:sales_dp}) to compute the expectation $G_{it}^-(\mu)$ for all $i \in \Ncal$ in $O((m n+T)\ts nT)$ operations. We simply skip time period $t$ in the dynamic program. To compute $H_{it}^-(\mu)$, we use $\mathbb P \{\sum_{k \in \Tcal \setminus \{t\}} \Brm_{ik}^\mu \geq c_i \} = \mathbb E\{ \min \{ c_i , \sum_{k \in \Tcal \setminus \{t\}} \Brm_{ik}^\mu \}\} - \mathbb E\{ \min \{ c_i -1, \sum_{k \in \Tcal \setminus \{t\}} \Brm_{ik}^\mu \}\}$, along with the dynamic program for computing $G_{it}^-(\mu)$. For each $j$ and $S \in \Omega_{jt}^\mu$, considering policy $\eta = \Mrm_{jt}(\mu,S)$, policy $\eta$ offers the  assortment $S$ to a customer of type $j$, so we compute $\varphi_{it}^\eta$ as $\varphi_{it}^\eta = \varphi_{it}^\mu - \sum_{Q \in \Omega_{jt}^\mu} \mathbb P \{ \Srm_{jt}^\mu = Q \} \ts \phi_{ij}(Q) + \phi_{ij}(S)$ in $O(n)$ operations.  Once we have $\varphi_{it}^\eta$, we use (\ref{eqn:rev_dec}) along with $G_{it}^-(\mu)$ and $H_{it}^-(\mu)$, to compute $\rev(\eta)$ in $O(n)$ operations. Thus, for fixed time period~$t$ and customer type $j$, we compute $\rev(\Mrm_{jt}(\mu,S))$ for all $S \in \Omega_{jt}^\mu$ in $O(n^2)$ operations, so we can execute the loop in our algorithm for fixed time period $t$ in $O((m n+T)\ts nT)$ operations. 



{\bf \underline{Considering Assortments Beyond the Support}:}
\\
\indent By Proposition \ref{pro:derand}, we have  \mbox{$\rev(\mu)  \leq\max_{S \in \Omega_{jt}^\mu} \rev(\Mrm_{jt}(\mu,S)) \leq \max_{S \subseteq \Ncal} \rev(\Mrm_{jt}(\mu,S))$}, so by considering an assortment $S$ beyond the support $\Omega_{jt}^\mu$, policy $\Mrm_{jt}(\mu,S)$ can potentially improve the total expected revenue of the policy obtained by focusing only on the assortments in the support $\Omega_{jt}^\mu$. Because $|\Omega_{jt}^\mu| = O(n)$, we can solve the problem $\max_{S \in \Omega_{jt}^\mu} \rev(\Mrm_{jt}(\mu,S))$ efficiently by checking the total expected revenue $\rev(\Mrm_{jt}(\mu,S))$ for each $S \in  \Omega_{jt}^\mu$. We show that we can also solve the problem $\max_{S \subseteq \Ncal} \rev(\Mrm_{jt}(\mu,S))$ efficiently as long as we can solve the so-called static assortment optimization problem under the choice model that governs the choices of the customers of each type efficiently. The characterization of the total expected revenue of a policy given in (\ref{eqn:rev_dec}) becomes critical for this purpose as well. Recall that the total expected revenue of policy $\mu$ is given by $\sum_{i \in \Ncal} r_i \ts G_i(\mu)$, where $G_i(\mu) = \mathbb E\{ \min\{ c_i , \sum_{t \in \Tcal} \Brm_{it}^\mu\}\}$ and $G_i(\mu)$ is equivalently given by (\ref{eqn:rev_dec}). We use $G_{it}^-(\mu)$ and $H_{it}^-(\mu)$ to denote, respectively, the expectation and probability in (\ref{eqn:rev_dec}). Policy $\Mrm_{jt}(\mu,S)$ offers the same assortment as policy $\mu$ at all time periods other than time period $t$, which implies that we have $G_{it}^-(\Mrm_{jt}(\mu,S)) = G_{it}^-(\mu)$ and $H_{it}^-(\Mrm_{jt}(\mu,S)) = H_{it}^-(\mu)$. In this case, noting that we defined \mbox{$\varphi_{it}^\mu = \sum_{j \in \Jcal}  \lambda_{jt} \sum_{Q \in \Omega_{jt}^\mu} \mathbb P \{ \Srm_{jt}^\mu = Q \} \ts \phi_{ij}(Q)$}, because policy $\Mrm_{jt}(\mu,S)$ offers the deterministic assortment $S$ to a customer of type $j$ at time period~$t$, using the expression in (\ref{eqn:rev_dec}) for policy $\Mrm_{jt}(\mu,S)$, we can compute $G_i(\Mrm_{jt}(\mu,S))$ as 
\begin{align*}
G_i(\Mrm_{jt}(\mu,S)) ~=~ G_{it}^-(\mu) +  \Bigg( \lambda_{jt} \ts \phi_{ij}(S) ~+ \!\!\! \sum_{k \in \Jcal \setminus \{j\}} \!\!\!\! \lambda_{kt} \!\! \sum_{Q \in \Omega_{kt}^\mu} \mathbb P \{ \Srm_{kt}^\mu = Q \} \ts \phi_{ik}(Q) \Bigg) \ts H_{it}^-(\mu),
\end{align*}
where we use $G_{it}^-(\Mrm_{jt}(\mu,S)) = G_{it}^-(\mu)$ and $H_{it}^-(\Mrm_{jt}(\mu,S)) = H_{it}^-(\mu)$, as well as the fact that policies $\Mrm_{jt}(\mu,S)$ and $\mu$ offer the same assortments to customers of type other than $j$. The total expected revenue of policy $\Mrm_{jt}(\mu,S)$ is $\rev(\Mrm_{jt}(\mu,S)) = \sum_{i \in \Ncal} r_i \ts G_i(\Mrm_{jt}(\mu,S))$, so multiplying the expression above with $r_i$ and adding over all $i \in \Ncal$, the total expected revenue of policy $\Mrm_{jt}(\mu,S)$ depends on the assortment $S$ only through $\lambda_{jt} \ts \sum_{i \in \Ncal} r_i \ts H_{it}^-(\mu) \ts \phi_{ij}(S)$. Thus, to obtain an optimal solution to the problem $\max_{S \subseteq \Ncal} \rev(\Mrm_{jt}(\mu,S))$, we can solve the problem $\max_{S \subseteq \Ncal} \sum_{i \in \Ncal} r_i \ts H_{it}^-(\mu) \ts \phi_{ij}(S)$. The last problem is the static assortment optimization problem under the choice model that governs the choices of customers of type $j$. In this problem, if we offer the assortment $S$ to a customer of type $j$, then the customer chooses product $i$ with probability $\phi_{ij}(S)$, resulting in an \underline{\it adjusted} revenue of $r_i \ts H_{it}^-(\mu)$. There are efficient algorithms for the static assortment optimization problem under various choice models, such as the multinomial, nested logit and Markov chain choice models; see \cite{TaRy04}, \cite{DaGa11} and \cite{BlGa13}. 
De-randomizing using assortments beyond the support of the original policy is guaranteed to perform at least as well as the original policy, but it is not guaranteed to outperform focusing only on the support. 

\vspace{-1mm}

\section{Performance Guarantees for Locally-Optimal Policies}
\label{sec:local_opt}

\vspace{-1mm}

In the previous section, our understanding is that we have a sampling-based inventory-agnostic policy with a certain performance guarantee and we want to find a policy that offers a deterministic sequence of assortments without degrading the original sampling-based inventory-agnostic policy. We turn our attention to a more ambitious question, where we start with an inventory-agnostic policy that offers a deterministic sequence of assortments, but the policy does not necessarily have a performance guarantee. We carry out local exchanges on the assortments offered by the policy by replacing the assortment offered to one customer type at one time period with an assortment to improve the total expected revenue of the policy. If no local exchanges improve the total expected revenue of the policy, then we say that we have a locally-optimal policy. Even when the policy that we start with has no performance guarantee, we give a performance guarantee for the \mbox{locally-optimal} policy. The performance guarantee holds for any locally-optimal policy that we obtain. Considering inventory-agnostic policy $\mu$ that offers a deterministic sequence of assortments, we say that this policy is \underline{\it locally-optimal with threshold $\delta$} if  the total expected revenue of the policy satisfies $\rev(\Mrm_{jt}(\mu,S)) \leq \rev(\mu) + \delta$ for all $j \in \Jcal$, $t \in \Tcal$ and $S \subseteq \Ncal$. In other words, if we cannot improve the total expected revenue of a policy by more than $\delta$ by exchanging the assortment offered by the policy to some customer type at some time period, then we say that the  policy is \mbox{locally-optimal} with threshold $\delta$.  In the next theorem, we give our main result for a locally-optimal policy with threshold $\delta$, where we give a performance guarantee for such a policy. We devote the second half of this section to the proof of the theorem. In the next theorem and throughout the rest of this section, we use $\opt^\text{\sf det}$ to denote the total expected revenue of the best inventory-agnostic policy that offers a deterministic sequence of assortments. Because we can obtain an \mbox{inventory-agnostic} policy that offers a deterministic sequence of assortments without degrading the total expected revenue of any sampling-based inventory-agnostic policy, $\opt^\text{\sf det}$ is also  the total expected revenue of the best sampling-based inventory-agnostic policy. 

\vspace{-1mm}

\begin{thm}[Performance of Any Locally-Optimal Policy]
\label{thm:loc_opt}
If policy $\mu$ is locally-optimal with threshold $\delta$, then we have $\rev(\mu) \geq \frac12 \opt^\text{\sf det} - \delta \frac{mT}{2}$. In particular, any locally-optimal policy with threshold zero is a $\frac 12$-approximate policy with respect to $\opt^\text{\sf det}$.

\end{thm}

\vspace{-1mm}

Below is a corollary to the theorem above that allows us to construct an algorithm to compute a locally-optimal policy with a performance guarantee in polynomial time by starting from an arbitrary policy. We discuss such an algorithm right after the corollary. 

\vspace{-1mm}

\begin{corollary} [Half-Approximation] For some $\epsilon > 0$, if an inventory-agnostic policy $\mu$ that offers a  deterministic sequence of assortments satisfies
$
\rev(\Mrm_{jt}(\mu,S)) \ts \leq \ts (1 + \epsilon \ts \frac{4}{mT}) \ts \rev(\mu)
$
for all $j \in \Jcal$, $t \in \Tcal$ and $S \subseteq \Ncal$, then we have $\rev(\mu) \geq (\frac12 - \epsilon) \ts \opt^\text{\sf det}$.

\end{corollary}

\vspace{-1mm}

\noindent{\it Proof.} By the assumption in the corollary, policy $\mu$ is locally-optimal with threshold $\epsilon \ts \frac{4}{mT}\ts \rev(\mu)$, in which case, by Theorem \ref{thm:loc_opt}, we have $\rev(\mu) \geq \frac 12 \ts \opt^\text{\sf det} - 2 \epsilon \ts \rev(\mu)$. Arranging the terms in the last inequality yields $\rev(\mu) \geq \frac{1}{2(1+2\epsilon)} \ts \opt^\text{\sf det} \geq \frac{1}{2} (1 - 2 \epsilon) \opt^\text{\sf det} = (\frac12 - \epsilon) \ts \opt^\text{\sf det}$. \qed

Consider a local exchange algorithm, where we start with an inventory-agnostic policy that offers a deterministic sequence of assortments. Given that the current policy is $\mu$, for each $j \in \Jcal$ and  $t \in \Tcal$, letting $S_{jt}^* = \arg\max_{S \subseteq \Ncal} \rev(\Mrm_{jt}(\mu,S))$, we check whether the total expected revenue of policy $\Mrm_{jt}(\mu,S_{jt}^*)$ improves that of policy $\mu$ by a factor of at least $1 + \epsilon \ts \frac{4}{mT}$. By the discussion at the end of the previous section, recall that we can solve the problem $\max_{S \subseteq \Ncal} \rev(\Mrm_{jt}(\mu,S))$ efficiently, as long as we can efficiently solve the corresponding static assortment optimization problem. For some $j \in \Jcal$ and $t \in \Tcal$, if we can find a policy $\Mrm_{jt}(\mu,S_{jt}^*)$ that improves the total expected revenue of policy $\mu$ by a factor of at least $1 + \epsilon \ts \frac{4}{mT}$, then we redefine policy $\mu$ as $\Mrm_{jt}(\mu,S_{jt}^*)$ and start over. The total expected revenue of the policy improves by a factor of $1 + \epsilon \ts \frac{4}{mT}$ at each update. Also, the total expected revenue of any policy is at most $\sum_{i \in \Ncal} r_i \ts c_i$. Thus, using $\Rbar$ to denote the total expected revenue on the initial policy,  the algorithm stops in $O(\frac{mT}{\epsilon} \log(\frac{\sum_{i \in \Ncal} r_i \ts c_i}{\Rbar}))$ updates. In the rest of this section, we turn our attention to giving a proof for Theorem \ref{thm:loc_opt}.

\vspace{1mm}

{\bf \underline{Proof of Theorem \ref{thm:loc_opt}\phantom{p}\!\!\!}:}
\\
\indent The outline of our analysis is as follows. Focusing on an inventory-agnostic policy that offers a deterministic sequence of assortments, we express the total expected revenue of the policy as a function of the expected demand for each product at each time period. We show that this function is differentiable, monotone and affine when viewed as a function of the expected demand for each product at each time period, as well as continuous DR-submodular. These properties allow us to use a  result for continuous DR-submodular functions, which shows that the approximately stationary points provide a $\frac 12 - \epsilon$ performance guarantee when we maximize these functions. Lastly, we show that if a policy is locally-optimal with threshold $\delta$, then the total expected revenue of this policy, when viewed as a function of the expected demand for each product at each time period, is an approximately stationary point. To execute this outline, fix an inventory-agnostic policy $\mu$ that offers the deterministic assortment $\Srm_{jt}^\mu$ to a customer of type $j$ at time period $t$. Under this policy, the demand for product $i$ at time period $t$ is Bernoulli with parameter $\varphi_{it}^\mu = \sum_{j \in \Jcal} \lambda_{jt} \ts \phi_{ij}(\Srm_{jt}^\mu)$. We use $\Brm_{it}(u)$ to denote a Bernoulli random variable with parameter $u$ with the understanding that the random variables $\{ B_{it}(u) : t \in \Tcal\}$ are independent of each other. For product $i$, we define
\begin{align}
F_i(u_1,\ldots,u_T) = \mathbb  E\Bigg\{ \min \Big\{ c_i , \sum_{t \in \Tcal} \Brm_{it}(u_t) \Big\}\Bigg\}.
\label{eqn:sales_prob}
\end{align}

\vspace{0.5mm}

Noting the argument just before the dynamic program in (\ref{eqn:sales_dp}), we can use the Bernoulli random variable $\Brm_{it}(\varphi_{it}^\mu)$ to capture the demand for product $i$ at time period $t$ under policy $\mu$, in which case,  the total expected revenue of policy~$\mu$ is given by $\sum_{i \in \Ncal} r_i \ts F_i(\varphi_{i1}^\mu,\ldots,\varphi_{iT}^\mu)$. 
In this case, to capture all possible sampling-based inventory-agnostic policies that we can use for a customer of a particular type at a particular time period, using the vector \mbox{$\alphavec =(\alpha(S) : S \subseteq \Ncal) \in \mathbb R_+^{2^n}$}, we define the unit simplex $\Delta = \{ \alphavec\in \mathbb R_+^{2^n} :  \sum_{S \subseteq \Ncal} \alpha(S) = 1 \}$. Therefore, using the vector $\pvec = (p_1,\ldots,p_n) \in \mathbb R_+^n$, the set of achievable purchase probabilities for the products at time period $t$ under any sampling-based inventory-agnostic policy are given by \mbox{$\Pcal_t = \{ \pvec \in \mathbb R_+^n: \exists \ts (\alphavec_1,\ldots,\alphavec_m) \in \Delta^m \mbox{ such that } p_i = \sum_{j \in \Jcal} \sum_{S \subseteq \Ncal} \lambda_{jt} \ts \phi_{ij}(S) \ts \alpha_j(S) ~ \forall \ts i \in \Ncal\}$}, where we use the fact that if we offer assortment $S$ to a customer of type $j$ at time period $t$ with probability $\alpha_j(S)$, then the purchase probability of product $i$ is $\sum_{j \in \Jcal} \sum_{S \subseteq \Ncal} \lambda_{jt} \ts \phi_{ij}(S) \ts \alpha_j(S)$. We argue that finding the best  \mbox{sampling-based} \mbox{inventory-agnostic} policy is equivalent to finding the best purchase probabilities for the products in $\Pcal_t$ for each time period $t$. In particular, consider any \mbox{sampling-based} inventory-agnostic policy $\mu$ that offers the assortment $S$ to a customer of type~$j$ at time period $t$ with probability $\mathbb P \{ \Srm_{jt}^\mu = S\}$. In this case, the purchase probability of product~$i$~at time period $t$ is  $\varphi_{it}^\mu = \sum_{j \in \Jcal}  \lambda_{jt} \sum_{S \in \Omega_{jt}^\mu} \mathbb P \{ \Srm_{jt}^\mu = S \} \ts \phi_{ij}(S)$. Therefore, setting \mbox{$\alpha_j(S) = \mathbb P \{ \Srm_{jt}^\mu = S\}$} for all $j \in \Jcal$ and $S \subseteq \Ncal$ in the definition of $\Pcal_t$, we have \mbox{$(\varphi_{1t}^\mu,\ldots,\varphi_{nt}^\mu) \in \Pcal_t$}. Conversely, consider any vector $\pvec \in \Pcal_t$. In this case, by the definition of $\Pcal_t$, there exists $(\alphavec_1,\ldots,\alphavec_m) \in \Delta^m$  such that \mbox{$p_i = \sum_{j \in \Jcal} \sum_{S \subseteq \Ncal} \lambda_{jt} \ts \phi_{ij}(S) \ts \alpha_j(S)$} for all $i \in \Ncal$. Therefore, if we use the sampling-based \mbox{inventory-agnostic} policy that offers assortment $S$ to a customer type $j$ at time period $t$ with probability $\alpha_j(S)$, then the purchase probability of product $i$ at time period $t$ under this policy is~$p_i$. By the preceding discussion, finding the best sampling-based inventory-agnostic policy is equivalent to finding the combination of best purchase probabilities for the products in $\Pcal_t$ for each time period~$t$.~In other words, using the vector $\pvec_t = (p_{1t},\ldots,p_{nt}) \in \mathbb R_+^n$, we have 
\begin{align}
\opt^\text{\sf det} ~~~~= \max_{(\pvec_1,\ldots,\pvec_T) \in \Pcal_1 \times \ldots \times \Pcal_T} \Bigg\{ \sum_{i \in \Ncal} r_i \ts F_i(p_{i1},\ldots,p_{iT}) \Bigg\}.
\label{eqn:opt_probs}
\end{align}

\vspace{-1mm}

{\it \underline{Monotone and Continuous DR-Submodular Functions\phantom{p}\!\!\!}.} Considering the box $\Xcal \subseteq \mathbb R_+^d$, the differentiable function $G : \Xcal \rightarrow \mathbb R_+$ is called continuous DR-submodular if \mbox{$\nabla G(\xvec) \leq \nabla G(\yvec)$} for all $\xvec, \yvec \in \Xcal$ with $\xvec \geq \yvec$, where we compare all vectors component-wise. On the other hand, the function $G : \Xcal \rightarrow \mathbb R_+$ is called monotone if $G(\xvec) \geq G(\yvec)$ for all $\xvec,\yvec \in \Xcal$ with $\xvec \geq \yvec$. Considering the~function $F_i : \mathbb [0,1]^T \rightarrow \mathbb R_+$~in (\ref{eqn:sales_prob}), we can use a coupling argument to show that $F_i$ is monotone. For $\uvec = (u_1,\ldots,u_T) \in [0,1]^T$ and $\wvec = (w_1,\ldots,w_T) \in [0,1]^T$ with $\uvec \geq \wvec$, using $U_t$ to denote a uniformly random variable over $[0,1]$ such that $U_1,\ldots,U_T$ are independent, we define the Bernoulli random variables $\Brm_{it}(u_t)$ and $\Brm_{it}(w_t)$ as $\Brm_{it}(u_t) = \ind{U_t \leq u_t}$ and $\Brm_{it}(w_t) = \ind{U_t \leq w_t}$, so we have $\Brm_{it}(u_t) \geq \Brm_{it}(w_t)$ almost surely. In this case, because $\min\{ c_i,x\}$ is increasing in $x$, $F_i$ is monotone. In the next lemma, we show that $F_i$ is also differentiable and continuous DR-submodular.

\vspace{-2mm}

\begin{lem} [Continuous DR-Submodularity]
\label{lem:drsubm}

The function $F_i : \mathbb [0,1]^T \rightarrow \mathbb R_+$ is differentiable and continuous DR-submodular. Furthermore, $F_i(u_1,\ldots,u_T)$ is linear in each $u_t$ for all $t \in \Tcal$. 

\end{lem}

\vspace{-2mm}

\noindent{\it Proof.} In (\ref{eqn:rev_dec}), the random variable $\Brm_{it}^\mu$ is Bernoulli with parameter $\varphi_{it}^\mu$. Using the same argument used to obtain (\ref{eqn:rev_dec}), we can express $F_i$ in (\ref{eqn:sales_prob}) equivalently as 
\begin{align}
F_i(u_1,\ldots,u_T) = \mathbb E\Bigg\{ \min \Big\{ c_i ,\!\!\! \sum_{k \in \Tcal \setminus \{t\}} \Brm_{ik}(u_k) \Big\}\Bigg\} + u_t \ts\ts \mathbb P \Bigg\{\sum_{k \in \Tcal \setminus \{t\}} \Brm_{ik}(u_k) \leq c_i - 1 \Bigg\}.
\label{eqn:rev_dec_probs}
\end{align}
Thus, we have $\frac{\partial F_i(u_1,\ldots,u_T)}{\partial u_t} = \mathbb P \{\sum_{k \in \Tcal \setminus \{t\}} \Brm_{ik}(u_k) \leq c_i - 1 \}$. By the coupling argument just before the lemma, for $\uvec = (u_1,\ldots,u_T) \in [0,1]^T$ and \mbox{$\wvec = (w_1,\ldots,w_T) \in [0,1]^T$} with $\uvec \geq \wvec$, we can define the Bernoulli random variables $\Brm_{ik}(u_k)$ and $\Brm_{ik}(w_k)$ such that \mbox{$\Brm_{ik}(u_k) \geq \Brm_{ik}(w_k)$} almost surely, so \mbox{$\ind{\sum_{k \in \Tcal \setminus\{t\}} \Brm_{ik}(u_k) \leq c_i -1} \leq \ind{\sum_{k \in \Tcal \setminus\{t\}} \Brm_{ik}(w_k) \leq c_i -1}$} almost surely. Taking expectations in the last inequality, we get \mbox{$\mathbb P \{\sum_{k \in \Tcal \setminus \{t\}} \Brm_{ik}(u_k) \leq c_i - 1 \} \leq \mathbb P \{\sum_{k \in \Tcal \setminus \{t\}} \Brm_{ik}(w_k) \leq c_i - 1 \}$}, so it follows that we have $\nabla F_i(u_1,\ldots,u_T) \leq \nabla F_i(w_1,\ldots,w_T)$ as well whenever $\uvec \geq \wvec$. Thus, $F_i$ is differentiable and continuous DR-submodular. By (\ref{eqn:rev_dec_probs}), $F_i(u_1,\ldots,u_T)$ is linear in $u_t$. \qed

{\it \underline{Approximate Stationarity and Performance Guarantee}.} By the discussion just before (\ref{eqn:opt_probs}), using the vector \mbox{$\pvec_t = (p_{1t},\ldots,p_{nt}) \in [0,1]^n$}, if we have $(\pvec_1,\ldots,\pvec_T) \in \Pcal_1 \times \ldots \times \Pcal_T$, then there exists a sampling-based inventory-agnostic policy such that the demand probability of product $i$ at time period $t$ is given by $p_{it}$ for all $i \in \Ncal$ and $t \in \Tcal$. Conversely, if $p_{it}$ is the demand probability of product $i$ at time period $t$ for all $i \in \Ncal$ and $t \in \Tcal$ under a sampling-based inventory-agnostic policy, then we have $(\pvec_1,\ldots,\pvec_T) \in \Pcal_1 \times \ldots \times \Pcal_T$. Letting $R(\pvec_1,\ldots,\pvec_T) = \sum_{i \in \Ncal} r_i \ts  F_i(p_{i1},\ldots,p_{iT})$, we use $R(\pvec_1,\ldots,\pvec_T)$ to capture the total expected revenue when we use a policy with a demand probability of $p_{it}$ for product $i$ at time period $t$ for all $i \in \Ncal$ and $t \in \Tcal$. Because $F_i: \mathbb [0,1]^T \rightarrow \mathbb R_+$ is monotone and continuous DR-submodular for all $i \in \Ncal$, $R : [0,1]^{n \times T} \rightarrow \mathbb R_+$ is also monotone and continuous DR-submodular. We use $\nabla_t R(\pvec_1,\ldots,\pvec_T)$ to denote the gradient block of $R(\pvec_1,\ldots,\pvec_T)$ for time period $t$, so $\nabla_t R(\pvec_1,\ldots,\pvec_T) = (\frac{\partial R(\pvec_1,\ldots,\pvec_T)}{\partial p_{it}} : i \in \Ncal)$. A point $(\pvechat_1,\ldots,\pvechat_T) \in \Pcal_1 \times \ldots \times \Pcal_T$ is called an \underline{\it $\epsilon$-stationary point} with respect to  $\Pcal_1 \times \ldots \times \Pcal_T$ if $\sum_{t \in \Tcal} \langle \nabla_t R(\pvechat_1,\ldots,\pvechat_T) , \qvec_t - \pvechat_t \rangle \leq \epsilon$  for all $(\qvec_1,\ldots,\qvec_T) \in \Pcal_1 \times \ldots \times \Pcal_T$. This definition of an $\epsilon$-stationary point follows Hassani et al. (2017). In the next lemma, we give a useful property of monotone and continuous DR-submodular functions. This property corresponds to Inequality (7.2) in \cite{HaSo17}.

\vspace{-1mm}

\begin{lem}[Implication of Continuous DR-Submodularity] 
\label{lem:var_drsubm}
For the box $\Xcal \subseteq \mathbb R_+^d$, letting $G: \Xcal \rightarrow \mathbb R_+$ be monotone and continuous DR-submodular, for all $\xvec ,\yvec \in \Xcal$, we have 
\begin{align*}
G(\yvec) - 2 \ts G(\xvec) ~\leq~ \langle \nabla G(\xvec) , \yvec - \xvec \rangle.
\end{align*}
\end{lem}

\vspace{-1mm}

In the next lemma, we use the lemma above to show that any $\epsilon$-stationary point of $R$ with respect to  $\Pcal_1 \times \ldots \times \Pcal_T$ yields a half-approximate solution to (\ref{eqn:opt_probs}) up to an additive error of $O(\epsilon)$.

\vspace{-1mm}

\begin{lem}[Half-Approximation via Stationarity]
\label{lem:half_stat}
 If $(\pvechat_1,\ldots,\pvechat_T)$ is an \mbox{$\epsilon$-stationary} point of $R$ with respect to  $\Pcal_1 \times \ldots \times \Pcal_T$, then we have $R(\pvechat_1,\ldots,\pvechat_T) \geq \frac 12 \ts \opt^\text{\sf det} - \frac \epsilon 2$.

\end{lem}

\vspace{-1mm}

\noindent{\it Proof.} Let $(\pvec_1^*,\ldots,\pvec_T^*)$ be an optimal solution to problem (\ref{eqn:opt_probs}). By Lemma~\ref{lem:drsubm}, \mbox{$R: [0,1]^{n \times T} \rightarrow \mathbb R_+$} is monotone and continuous DR-submodular, so by~Lemma \ref{lem:var_drsubm}, we get $\opt^\text{\sf det} - 2 \ts R(\pvechat_1,\ldots,\pvechat_T) = R(\pvec_1^*,\ldots,\pvec_T^*) - 2 \ts R(\pvechat_1,\ldots,\pvechat_T) \leq \sum_{t \in \Tcal} \langle \nabla_t R(\pvechat_1,\ldots,\pvechat_T) , \pvec_t^* - \pvechat_t \rangle  \leq  \epsilon$, where the first inequality holds because the inner product is additive and the second inequality is by $\epsilon$-stationarity. The desired result follows by arranging the terms in the last chain of inequalities. \qed

In the next lemma, we relate local-optimality to $\epsilon$-stationarity. We let $\varphi_{it}^\mu$ be the demand probability of product $i$ at time period $t$ under policy $\mu$. We use the vector $\phivec_t^\mu = (\varphi_{1t}^\mu,\ldots,\varphi_{nt}^\mu)$.

\vspace{-1mm}

\begin{lem}[Stationarity via Local Optimality]
\label{lem:stat_opt}
If policy $\mu$ is locally-optimal with threshold $\delta$, then the point $(\phivec_1^\mu,\ldots,\phivec_T^\mu)$ is $\epsilon$-stationary for $R : [0,1]^{n \times T} \rightarrow \mathbb R_+$  with $\epsilon = mT\delta$.

\end{lem}

\vspace{-2mm}

\noindent{\it Proof.} We use $\Srm_{jt}^\mu$ to denote the deterministic assortment offered by policy $\mu$ to a customer of type~$j$ at time period $t$. In this case, the demand probability for product $i$ at time period $t$ is given by \mbox{$\varphi_{it}^\mu = \sum_{j \in \Jcal} \lambda_{jt} \ts \phi_{ij}(\Srm_{jt}^\mu)$}. Using the notation $\qvechat_t = (\qhat_{1t},\ldots,\qhat_{nt})$, consider any \mbox{$(\qvechat_1,\ldots,\qvechat_T) \!\in\! \Pcal_1 \!\! \times \! \ldots \! \times\!\! \Pcal_T$}. Because $\qvechat_t \in \Pcal_t$, by the definition of $\Pcal_t$, there exists $\alphavechat_{jt} = (\alphahat_{jt}(S) : S \subseteq \Ncal) \in \Delta$ such that we have $\qhat_{it} = \sum_{j \in \Jcal} \lambda_{jt} \sum_{S \subseteq \Ncal} \alphahat_{jt}(S) \ts \phi_{ij}(S)$ for all $i \in \Ncal$. For all $t \in \Tcal$, if we can show that $\langle \nabla_t R(\phivec_1^\mu,\ldots,\phivec_T^\mu) , \qvechat_t - \phivec_t^\mu \rangle \leq m \delta$, then we get $\sum_{t \in \Tcal} \langle \nabla_t R(\phivec_1^\mu,\ldots,\phivec_T^\mu) , \qvechat_t - \phivec_t^\mu \rangle \leq mT \delta$, so the desired result follows. We turn our attention to showing that $\langle \nabla_t R(\phivec_1^\mu,\ldots,\phivec_T^\mu) , \qvechat_t - \phivec_t^\mu \rangle \leq m \delta$ for all $t \in \Tcal$.
Note that policy $\mu$ is locally-optimal with threshold $\delta$, so the total expected revenue of policy $\Mrm_{jt}(\mu,S)$ does not improve that of policy $\mu$ by more than $\delta$ for any $S \subseteq \Ncal$. 

Fix time period $t$, customer type $j$ and assortment $S$.  Total expected revenue of policy $\mu$ is $R(\phivec_1^\mu,\ldots,\phivec_T^\mu)$. Considering policy $\pi(j,t,S) = \Mrm_{jt}(\mu,S)$, because policy $\mu$ is locally-optimal, we have $R(\phivec_1^{\pi(j,t,S)},\ldots,\phivec_T^{\pi(j,t,S)})- R(\phivec_1^\mu,\ldots,\phivec_T^\mu) \leq \delta$.  The demand probability of product $i$ at time period $t$ under policy $\pi(j,t,S)$ is $\varphi_{it}^{\pi(j,t,S)} = \lambda_{jt} \ts \phi_{ij}(S) + \sum_{k \in \Jcal \setminus \{j\}} \lambda_{kt} \ts \phi_{ik}(\Srm_{kt}^\mu)$, where we use the fact that policies $\mu$ and $\pi(j,t,S)$ differ only in the assortment offered to customers of type $j$ at time period $t$.  Therefore, we have  $\varphi_{it}^{\pi(j,t,S)} - \varphi_{it}^\mu = \lambda_{jt} \ts (\phi_{ij}(S) - \phi_{ij}(\Srm_{jt}^\mu))$ for all $i \in \Ncal$. By Lemma \ref{lem:drsubm}, $F_i(p_{i1},\ldots, p_{iT})$ is linear in $p_{it}$, in which case, noting that $R(\pvec_1,\ldots,\pvec_T) = \sum_{i \in \Ncal} r_i \ts  F_i(p_{i1},\ldots,p_{iT})$, it follows that $R(\pvec_1,\ldots,\pvec_T)$ is linear in $\pvec_t$ as well. In this case, using the linearity of $R(\pvec_1,\ldots,\pvec_T)$ in $\pvec_t$, as well as the fact that policy $\mu$ is locally-optimal with threshold $\delta$, we obtain 
\begin{align}
\big \langle \nabla_t R(\phivec_1^\mu,\ldots,\phivec_T^\mu) , \phivec_t^{\pi(j,t,S)} - \phivec_t^\mu \big \rangle
~=~
R(\phivec_1^{\pi(j,t,S)},\ldots,\phivec_T^{\pi(j,t,S)})- R(\phivec_1^\mu,\ldots,\phivec_T^\mu) 
~\leq~
\delta.
\label{eqn:stat_rev}
\end{align} 
\indent Noting that $\varphi_{it}^{\pi(j,t,S)} - \varphi_{it}^\mu = \lambda_{jt} \ts (\phi_{ij}(S) - \phi_{ij}(\Srm_{jt}^\mu))$, using the vector $\Phi_j(S) = (\phi_{1j}(S),\ldots,\phi_{nj}(S))$,  multiplying (\ref{eqn:stat_rev}) by $\alphahat_{jt}(S)$ and adding over all $S \subseteq \Ncal$ and $j \in \Jcal$, we obtain
\begin{align*}
 m \ts \delta 
~&\stackrel{(a)}=~ 
\sum_{S \in \Ncal} \sum_{j \in \Jcal}  \alphahat_{jt}(S) \ts \delta 
~\stackrel{(b)}\geq~
\Big \langle \nabla_t R(\phivec_1^\mu,\ldots,\phivec_T^\mu) , \sum_{S \subseteq \Ncal} \sum_{j \in \Jcal} \alphahat_{jt}(S) \ts (\phivec_t^{\pi(j,t,S)} - \phivec_t^\mu) \Big \rangle
\\
&\stackrel{(c)}=~
\Big \langle \nabla_t R(\phivec_1^\mu,\ldots,\phivec_T^\mu) , \sum_{S \subseteq \Ncal} \sum_{j \in \Jcal} \alphahat_{jt}(S) \ts \lambda_{jt} \ts (\Phi_j(S) - \Phi_j(\Srm_{jt}^\mu)) \Big \rangle
~\stackrel{(d)}=~
\langle \nabla_t R(\phivec_1^\mu,\ldots,\phivec_T^\mu) , \qvechat_t - \phivec_t^\mu \rangle,
\end{align*}
where $(a)$ uses $\sum_{S \subseteq \Ncal} \alphahat_{jt}(S) = 1$, $(b)$ is by (\ref{eqn:stat_rev}), $(c)$ uses $\varphi_{it}^{\pi(j,t,S)} - \varphi_{it}^\mu = \lambda_{jt} \ts (\phi_{ij}(S) - \phi_{ij}(\Srm_{jt}^\mu))$ and $(d)$ holds because $\qhat_{it} = \sum_{j \in \Jcal} \lambda_{jt} \sum_{S \subseteq \Ncal} \alphahat_{jt}(S) \ts \phi_{ij}(S)$ and $\varphi_{it}^\mu = \sum_{j \in \Jcal} \lambda_{jt} \ts \phi_{ij}(\Srm_{jt}^\mu)$.  \qed

By Lemmas \ref{lem:half_stat} and \ref{lem:stat_opt}, a half-approximate policy is related to an approximate stationary point, which is, in turn, related to a locally-optimal policy. We give a proof of Theorem \ref{thm:loc_opt} next.

{\it \underline{Proof of Theorem \ref{thm:loc_opt}}.} By Lemma \ref{lem:stat_opt}, the point $(\phivec_1^\mu,\ldots,\phivec_T^\mu)$ is $mT\delta$-stationary, in which case, by Lemma~\ref{lem:half_stat}, the total expected revenue of policy $\mu$ is at least $\frac12 \opt^\text{\sf det} - \delta \frac{mT}{2}$. \qed

\vspace{-2mm}

\section{Lossless Inventory-Aware Execution}
\label{sec:inv_aware}

\vspace{-2.5mm}

A sampling-based inventory-agnostic policy may sample an assortment that includes products without remaining inventories. If the customer chooses a product without remaining inventories, then she leaves without a purchase. Offering products without remaining inventories may not be meaningful in practice. In this section, assuming that the customers choose according to  a choice model satisfying the substitutability property, we show that if we are given a sampling-based inventory-agnostic policy, then we can implement this policy in a way that never offers products without remaining inventories while we do not degrade the performance of the policy. In particular, all we need to do is to drop the products without remaining inventories and offer the remaining products. In this way, we obtain a \underline{\it sampling-based inventory-aware} policy that never offers products without remaining inventories. Given a sampling-based inventory-agnostic policy $\mu$, below is the description of the sampling-based inventory-aware counterpart of policy $\mu$.

\vspace{-0.5mm}

{\bf \underline{Sampling-Based Inventory-Aware Counterpart of a Policy}:}
\\
\indent Consider sampling-based inventory-agnostic policy $\mu$ that offers the random assortment $\Srm_{jt}^\mu$ to a customer of type $j$ at time period $t$. Assume that we are at time period $t$ and the remaining inventories of the products are given by $\xvec = (x_i : i \in \Ncal)$, where $x_i$ is the remaining inventory of product $i$. Letting $\Ncal(\xvec) = \{ i \in \Ncal : x_i > 0\}$ be the set of products with remaining inventories, if a customer of type $j$ arrives, then the sampling-based inventory-aware counterpart of policy $\mu$ samples assortment $S$ with probability $\mathbb P \{ \Srm_{jt}^\mu = S\}$ and offers the assortment $S \cap \Ncal(\xvec)$. 

\vspace{-0.5mm}

In the next proposition, we show that the total expected revenue of the sampling-based  inventory-aware counterpart of policy $\mu$ is at least as large as that of the total expected revenue of policy $\mu$. In the proof, we use the observation that if the choice model governing the choices of the customers satisfies the substitutability property, then we have the inequality $\phi_{ij}(S \cap \Ncal(\xvec))  \geq  \phi_{ij}(S) \ts \ind{x_i > 0}$ for all $i \in \Ncal$, $j \in \Jcal$, $S \subseteq \Ncal$ and $\xvec \in \mathbb Z_+^n$. In particular, using the fact that the choice probability of a product not included in the assortment is zero, if $i \not \in S$, then both sides of the inequality is zero, so the inequality holds. On the other hand, if $i \not \in \Ncal(\xvec)$, then we have $\ind{x_i > 0} = 0$ by the definition of $\Ncal(\xvec)$, so the right side of the inequality is zero and the left side of the inequality is non-negative, so the inequality holds. Lastly, if $i \in S \cap \Ncal(\xvec)$, then we have $\phi_{ij}(S \cap \Ncal(\xvec)) \geq \phi_{ij}(S)$ by the substitutability of the choice model.  Also, because $i \in S \cap \Ncal(\xvec)$, we have $\ind{x_i > 0} = 1$, in which case, we get $\phi_{ij}(S \cap \Ncal(\xvec)) \geq \phi_{ij}(S) =  \phi_{ij}(S) \ts \ind{x_i > 0} $, so the inequality holds.

\vspace{-1.5mm}

\begin{prop}[Inventory-Aware]
\label{pro:inv}
 Letting $\overline \rev(\mu)$ be the total expected revenue of the sampling-based  inventory-aware counterpart of policy $\mu$, we have $\overline \rev(\mu) \geq \rev(\mu)$. 
\end{prop}

\vspace{-2mm}

\noindent{\it Proof.} Let $\Srm_{jt}^\mu$ be the random assortment that sampling-based inventory-agnostic policy $\mu$ offers to a customer of type $j$ at time period $t$. Given that the remaining inventories of the products at time period $t$ correspond to the vector $\xvec$, we use $J_t^\mu(\xvec)$ to denote the total expected revenue of policy $\mu$ over time periods $t,\ldots,T$. We define $\Jbar_t^\mu(\xvec)$ similarly for the inventory-aware counterpart of policy $\mu$. We use induction over the time periods to show that  $\Jbar_t^\mu(\xvec) \geq J_t^\mu(\xvec)$ for all $\xvec \in \mathbb Z_+^n$ and $t \in \Tcal$. We have $\Jbar_{T+1}^\mu(\xvec) = 0 = J_{T+1}^\mu(\xvec)$ for all $\xvec \in \mathbb Z_+^n$, so the result holds at time period $T+1$. Assuming that the result holds at time period $t+1$, we show that the result holds at time period $t$ as well.~We can use a dynamic program to  compute the total expected revenue  of the inventory-aware counterpart of policy $\mu$ over time periods $t,\ldots,T$. Given that the remaining inventories of the products at time period $t$ correspond to the vector $\xvec$, the inventory-aware counterpart of policy $\mu$ offers the assortment $\Srm_{jt}^\mu \cap \Ncal(\xvec)$ to a customer of type $j$ at time period $t$, in which case, noting that all of the products in $\Srm_{jt}^\mu \cap \Ncal(\xvec)$ have remaining inventories, the customer purchases product $i$ with probability $\phi_{ij}(\Srm_{jt}^\mu \cap \Ncal(\xvec))$. Thus, letting $\evec_i \in \{0,1\}^n$ be the $i$-th unit vector, we have 
\begin{align*}
\Jbar_t^\mu(\xvec) ~&\stackrel{(a)}=~
\sum_{j \in \Jcal} \lambda_{jt} \ts \mathbb E  \Bigg\{ \sum_{i \in \Ncal}  \phi_{ij}(\Srm_{jt}^\mu \cap \Ncal(\xvec)) \ts \Big\{ r_i + \Jbar_{t+1}^\mu(\xvec - \evec_i) \Big\} 
+
\Big\{ 1 -  \sum_{i \in \Ncal}  \phi_{ij}(\Srm_{jt}^\mu \cap \Ncal(\xvec)) \Big\} \ts \Jbar_{t+1}^\mu(\xvec) \Bigg\}
\\
~&\stackrel{(b)}\geq~
\sum_{j \in \Jcal} \lambda_{jt} \ts \mathbb E  \Bigg\{ \sum_{i \in \Ncal}  \phi_{ij}(\Srm_{jt}^\mu \cap \Ncal(\xvec)) \ts \Big\{ r_i + J_{t+1}^\mu(\xvec - \evec_i) \Big\} 
+
\Big\{ 1 -  \sum_{i \in \Ncal}  \phi_{ij}(\Srm_{jt}^\mu \cap \Ncal(\xvec)) \Big\} \ts J_{t+1}^\mu(\xvec) \Bigg\}
\\
~&\stackrel{(c)}=~
\sum_{j \in \Jcal} \lambda_{jt} \ts \mathbb E  \Bigg\{ \sum_{i \in \Ncal}  \phi_{ij}(\Srm_{jt}^\mu \cap \Ncal(\xvec)) \ts \Big\{ r_i + J_{t+1}^\mu(\xvec - \evec_i) - J_{t+1}^\mu(\xvec)\Big\} + J_{t+1}^\mu(\xvec) \Bigg\}
\\
~&\stackrel{(d)}\geq~
\sum_{j \in \Jcal} \lambda_{jt} \ts \mathbb E  \Bigg\{ \sum_{i \in \Ncal} \ind{x_i > 0} \ts \phi_{ij}(S) \ts \Big\{ r_i + J_{t+1}^\mu(\xvec - \evec_i) - J_{t+1}^\mu(\xvec)\Big\} + J_{t+1}^\mu(\xvec) \Bigg\}
\\
&\stackrel{(e)}=~
\sum_{j \in \Jcal} \lambda_{jt} \ts \mathbb E  \Bigg\{ \sum_{i \in \Ncal}  \phi_{ij}(\Srm_{jt}^\mu) \ts \ind{x_i > 0}  \Big\{ r_i + J_{t+1}^\mu(\xvec - \evec_i) \Big\} 
+
\Big\{ 1 -  \sum_{i \in \Ncal}  \phi_{ij}(\Srm_{jt}^\mu) \ts \ind{x_i > 0}  \Big\} \ts J_{t+1}^\mu(\xvec) \Bigg\}.
\end{align*}  

Here, $(a)$ uses the dynamic program outlined at the beginning of the proof and the expectation involves the random assortment $\Srm_{jt}^\mu$, $(b)$ uses the induction assumption, $(c)$ and $(e)$ are by arranging the terms and $(d)$ holds because $\phi_{ij}(S \cap \Ncal(\xvec)) \geq  \phi_{ij}(S) \ts \ind{x_i > 0} $ and we shortly show the claim that $J_{t+1}^\mu(\xvec) - J_{t+1}^\mu(\xvec - \evec_i) \leq r_i$ for any sampling-based inventory-agnostic policy $\mu$ as long as $\xvec \geq \evec_i$. Under policy $\mu$, we offer the random assortment $\Srm_{jt}^\mu$ to a customer of type $j$ at time period $t$, the customer chooses product $i$ with probability $\phi_{ij}(\Srm_{jt}^\mu)$ and we make a sale for the product when we have remaining inventory for the product. Thus, the last expression above is equal to $J_t^\mu(\xvec)$, yielding the desired result. To finish the proof, we  argue that $J_t^\mu(\xvec) - J_t^\mu(\xvec - \evec_i) \leq r_i$ for all $t \in \Tcal$ and $\xvec \in \mathbb Z_+^n$ with $\xvec \geq \evec_i$. By the discussion just before (\ref{eqn:sales_dp}), defining $G_{it}(x_i; \mu) = \mathbb E\{ \min \{ x_i , \sum_{k=t}^T \Brm_{ik}^\mu\}\}$, we have \mbox{$J_t^\mu(\xvec) = \sum_{i \in \Ncal} r_i \ts G_{it}(x_i; \mu)$}. Because $\min \{ x_i , \sum_{k=t}^T \Brm_{ik}^\mu\} - \min \{ x_i -1, \sum_{k=t}^T \Brm_{ik}^\mu\} \leq 1$ almost~surely, we get $G_{it}(x_i; \mu) - G_{it}(x_i-1; \mu) \leq 1$, so  $J_t^\mu(\xvec) - J_t^\mu(\xvec - \evec_i) = r_i \ts (G_{it}(x_i; \mu) - G_{it}(x_i-1; \mu)) \leq r_i$. \qed

Thus, the inventory-aware counterpart of a given sampling-based inventory-agnostic policy is guaranteed to perform at least as well as the given policy.

\vspace{-2mm}

\section{Numerical Experiments}
\label{sec:numerics}

\vspace{-2mm}

We use randomly generated test problems to quantify the gains from \mbox{de-randomization}. By Proposition~\ref{pro:derand}, de-randomizing a policy cannot degrade the performance of the original policy, but the de-randomized policy can potentially perform better than the original one. By Theorem~\ref{thm:loc_opt}, a locally-optimal policy has a performance guarantee when compared with the best sampling-based policy. These results work with the total expected revenue of the policy when implemented in an inventory-agnostic fashion. By Proposition \ref{pro:inv}, implementing the inventory-aware counterpart of an inventory-agnostic policy does not degrade its performance.  Our numerical experiments demonstrate that the practical performance gains from de-randomizing a policy can be substantial, especially in the lower tail of the total revenue distribution. Searching for assortments beyond the support of the original policy during de-randomization can deliver further improvements.

We have 54 test problems, each with $n=50$ products and $m=20$ customer types. All customer types choose according to the multinomial logit model. Half of the test problems have $T=200$ time periods in the selling horizon. The other half has $T=1000$ time periods. In Appendix \ref{sec:gen}, we give the full generative formulas for our test problems. 
We de-randomize two baseline policies. First, we solve the so-called choice-based deterministic linear program given in Appendix \ref{sec:cdlp} and randomize the assortment offered to each customer type according to an optimal solution. We call this the CDLP policy. Second, we find the myopic assortment to maximize the expected revenue from customer type $j$ as $S_j^{\text{\sf myo}} = \arg\max_{S \subseteq \Ncal} \sum_{i \in \Ncal} r_i \ts \phi_{ij}(S)$. Indexing the products in the order of decreasing revenues so that $r_1 \geq r_2 \geq \ldots \geq r_n$, under the multinomial logit model, it is known that we can set \mbox{$S_j^{\text{\sf myo}} = \{1,\ldots,k_j^*\}$} for some $k_j^* =1,\ldots,n$. We randomize the assortment offered to customer type $j$ uniformly over the assortments $\{ \{1,\ldots,k\} : 1 \leq k \leq k_j^*\}$. This policy has no performance guarantee and is poor on its own, but we use it precisely to test whether de-randomization can rescue a weak starting point. We call this the uniform-myopic policy.

When de-randomizing each of the two baseline policies, we use two methods. In the support method, we restrict the deterministic assortment offered to each customer type and time period pair to the support of the baseline policy. In the global method, we can choose any subset of products to offer to a customer type and time period pair. 
The optimal objective value of the choice-based deterministic linear program provides an upper bound on the optimal total expected revenue, so we use this upper bound to normalize all total expected revenues. We report three gain metrics for each baseline policy and de-randomization method pair. First, we measure the change in the performance under the inventory-agnostic execution before and after de-randomization, given by $\Delta_{\text{\sf agn}}=\frac{1}{\text{\sf UB}}(R_{\text{\sf agn}}^{\text{\sf det}}-R_{\text{\sf agn}}^{\text{\sf rand}})\times 100\%$, where $R_{\text{\sf agn}}^{\text{\sf rand}}$ and $R_{\text{\sf agn}}^{\text{\sf det}}$ are, respectively, the total expected revenues of the baseline and de-randomized policies under inventory-agnostic execution and \text{\sf UB} is the upper bound. Second, we measure a similar change under the inventory-aware execution, given by \mbox{$\Delta_{\text{\sf awr}}=\frac{1}{\text{\sf UB}}(R_{\text{\sf awr}}^{\text{\sf det}}-R_{\text{\sf awr}}^{\text{\sf rand}})\times 100\%$, where $R_{\text{\sf awr}}^{\text{\sf rand}}$} and $R_{\text{\sf awr}}^{\text{\sf det}}$ are, respectively, the total expected revenues of the baseline and de-randomized policies under inventory-aware execution. We compute the latter two total expected revenues by using Monte Carlo simulation with 1000 sample paths. Third, we measure the change in the tail performance under the inventory-aware execution before and after de-randomization, given by  $\Delta P_{\alpha}=\frac{1}{\text{\sf UB}} \ts (P_\alpha^{\text{\sf det}}- P_\alpha^{\text{\sf rand}})\times 100\%$, where $P_\alpha^{\text{\sf rand}}$ and $P_\alpha^{\text{\sf det}}$ are, respectively, the total expected revenues of the baseline and de-randomized policies under inventory-aware execution when we focus on the $\alpha\%$ of the sample paths with the smallest total revenues. The latter performance measures are also known as CVaR. 
 
\vspace{-0.5mm}

\begin{table}[t]
\vspace{-4mm}
\centering
\scriptsize
\setlength{\tabcolsep}{4pt}
\begin{tabular}{lrrrrrrrr}
\toprule
 & \multicolumn{4}{c}{$T = 200$} & \multicolumn{4}{c}{$T = 1000$} \\
\cmidrule(lr){2-5} \cmidrule(lr){6-9}
 & \multicolumn{2}{c}{CDLP} & \multicolumn{2}{c}{Uniform-myopic} & \multicolumn{2}{c}{CDLP} & \multicolumn{2}{c}{Uniform-myopic} \\
\cmidrule(lr){2-3} \cmidrule(lr){4-5} \cmidrule(lr){6-7} \cmidrule(lr){8-9}
 & Support & Global & Support & Global & Support & Global & Support & Global \\
\midrule
$\Delta_{\text{\sf agn}}$ & 2.0 & 9.7 & 19.3 & 17.5 & 2.5 & 8.3 & 20.4 & 18.8 \\
$\Delta_{\text{\sf awr}}$ & 2.7 & 11.2 & 20.3 & 19.2 & 2.9 & 7.0 & 16.7 & 16.2 \\
\midrule
$\Delta P_1$ & 3.7 & 17.0 & 24.5 & 22.3 & 4.0 & 12.4 & 20.5 & 19.0 \\
$\Delta P_5$ & 3.3 & 16.0 & 23.8 & 21.6 & 3.7 & 11.2 & 19.8 & 18.4 \\
\bottomrule
\end{tabular}
\caption{\!\!\!\!\!\!\!Improvement from de-randomization for each baseline policy and de-randomization method pair.}
\label{tab:main-short}

\vspace{-3ex}

\end{table}

\begin{table}[t]
\vspace{-2mm}
\centering
{\scriptsize
\setlength{\tabcolsep}{4pt}
\begin{tabular}{lrrrrrrrr}
\toprule
 & \multicolumn{4}{c}{$T = 200$} & \multicolumn{4}{c}{$T = 1000$} \\
\cmidrule(lr){2-5} \cmidrule(lr){6-9}
 & \multicolumn{2}{c}{CDLP} & \multicolumn{2}{c}{Uniform-myopic} & \multicolumn{2}{c}{CDLP} & \multicolumn{2}{c}{Uniform-myopic} \\
\cmidrule(lr){2-3} \cmidrule(lr){4-5} \cmidrule(lr){6-7} \cmidrule(lr){8-9}
 & Support & Global & Support & Global & Support & Global & Support & Global \\
\midrule
$R_{\text{\sf agn}}^{\text{\sf rand}}$ & \multicolumn{2}{c}{80.0} & \multicolumn{2}{c}{70.3} & \multicolumn{2}{c}{90.4} & \multicolumn{2}{c}{78.5} \\
$R_{\text{\sf awr}}^{\text{\sf rand}}$ & \multicolumn{2}{c}{85.3} & \multicolumn{2}{c}{76.0} & \multicolumn{2}{c}{92.8} & \multicolumn{2}{c}{83.1} \\
$R_{\text{\sf awr}}^{\text{\sf det}}$ & 87.9 & 96.5 & 96.3 & 95.1 & 95.6 & 99.8 & 99.8 & 99.3 \\
\bottomrule
\end{tabular}
\caption{\!\!\!\!\!\!Performance before and after de-randomization for each baseline policy and de-randomization method pair.}
\label{tab:perf-levels}
}
\vspace{-6ex}
\end{table}

In Table \ref{tab:main-short}, the left and right portions, respectively, focus on the test problems with $T = 200$ and $T=1000$ time periods. For each baseline policy and de-randomization method pair, we give each of the three improvement measures in the previous paragraph averaged over the 27 test problems with a particular number of time periods in the selling horizon. For the CDLP baseline policy, the support method yields somewhat modest improvements at 2.0\% to 2.9\%, but the global method raises the improvements significantly to 7.0\% to 11.2\%. For the uniform-myopic baseline policy, which is poor on its own, both the support and global methods yield substantial improvements at 16.2\% to 20.4\%. The results in Table \ref{tab:main-short} give the performance improvement from de-randomization, but they do not convey whether the de-randomized policies are good. In Table \ref{tab:perf-levels}, we give the total expected revenue for each baseline policy and de-randomization method pair, normalized by the upper bound and averaged over the 27 test problems. For $T=200$ and $T=1000$, the global method obtains de-randomized policies with, respectively,  96.5\% and 99.8\% of the upper bound. Even if the baseline policy is uniform-myopic with total expected revenue of only 78.5\% of the upper bound, we reach virtually optimal policies after de-randomization! The results are striking. A single round of local improvements performs well beyond the theoretical guarantees and yields nearly~optimal policies especially at $T=1000$. In Appendix \ref{sec:exp_details}, we give the details of our~numerical results.

\bibliographystyle{ormsv080}

\renewcommand*{\bibfont}{\footnotesize}
\renewcommand*{\bibfont}{\normalfont\footnotesize\linespread{1}\selectfont}
\setlength{\bibsep}{2.0pt}
\bibliography{references}

\newpage
\ECSwitch

\begin{center}
\vspace{-9mm}
{\large  \underline{Electronic Companion}: \break Killing the Case for Randomization in Dynamic Assortment Optimization} 
\\
June 24, 2026
\end{center}

\begin{APPENDICES}

\vspace{-2mm}

\section{Generation of Our Test Problems}
\label{sec:gen}

\vspace{-2mm}

We give the full generative procedure for our test problems. In all test problems, the number of products is $n=50$ and the number of customer types is $m=20$. All customer types choose according to the multinomial logit model. Under the multinomial logit model, a customer of type~$j$ associates the preference weight $v_{ij}$ with product $i$. If we offer the assortment $S$ to a customer of type $j$, then the customer purchases product $i \in S$ with probability \mbox{$\phi_{ij}(S) = \frac{v_{ij}}{1 + \sum_{k \in S} v_{kj}}$}. We use the following approach to generate the preference weights. For each customer type $j$, we draw the size of its consideration set  $L_j$  from $\mathrm{Unif}\{\lfloor 0.4\, n \rfloor,\ldots,\lfloor 0.8\, n \rfloor\}$, in which case, we draw the consideration set~$\Ccal_j$ of customer type~$j$ uniformly over all subsets of products with size $L_j$. Customers of type $j$ make a purchase among the products in $\Ccal_j$. We draw a raw preference weight $w_{ij}$ from $\mathrm{Unif}[0.5,2.0]$ for each product $i \in \Ccal_j$, whereas we set $w_{ij}=0$ for $i \notin \mathcal{C}_j$. In this case, the preference weight that a customer of type~$j$ associates with product $i$ is given by $v_{ij} = \frac{1-P_0}{P_0\sum_{k\in\mathcal{C}_j} w_{kj}} \ts w_{ij}$, where $P_0$ is a parameter that we vary. If we offer the full consideration set $\mathcal{C}_j$ to a  customer of type $j$, then the customer purchases some product with probability $\sum_{i \in \Ccal_j} \phi_{ij}(\Ccal_j) = \frac{(1-P_0) / P_0}{1 + (1-P_0)/P_0} = 1 - P_0$, so the parameter $P_0$ controls the propensity of the customers to leave without a purchase.

We use time-varying arrival probabilities for customers of different types. At each time period, we have one customer arrival. The probability that a customer of type $j$ arrives at time period $t$ is given by  $\lambda_{jt}=\frac{1}{\sum_{\ell \in \Jcal} \exp(-\kappa |t-T_\ell^{\text{\sf peak}}|)} \ts \exp(-\kappa |t-T_j^{\text{\sf peak}}|)$, where $T_j^{\text{\sf peak}}$ is the peak arrival time period for customer type $j$. The parameter $\kappa$ controls the rate of decline of the arrival probabilities as we move away from the peak arrival time period. We vary this parameter as well. We set up the peak arrival time periods in such a way that customer types with larger consideration sets receive earlier peak arrival time periods. This ordering puts significant pressure on \mbox{inventory-agnostic} policies in the sense that flexible customer types with larger consideration sets arrive while inventory is still abundant and consume products that pickier customer types depend on later in the selling horizon, stimulating the need for inventory-aware adjustments to the sampling-based policy. Concretely, we order the $m$ customer types by the sizes of their consideration sets in decreasing order and place their peak arrival time periods equally spaced across the selling horizon, so the type with the largest consideration set peaks at $t=1$ and the type with the smallest consideration set peaks at $t=T$. Revenue $r_i$ of each product $i$ is drawn independently from $\mathrm{Unif}[1,10]$. 

To come up with the initial inventories of the products, we compute the myopic assortment to maximize the expected revenue from customer type $j$ as $S_j^{\text{\sf myo}} = \arg\max_{S \subseteq \Ncal} \sum_{i \in \Ncal} r_i \ts \phi_{ij}(S)$. If we always offer the myopic assortments, then the total expected demand for product $i$ is given by $ d_i = \sum_{t \in \Tcal} \sum_{j \in \Jcal} \lambda_{jt} \ts \phi_{ij}(S_j^{\text{\sf myo}})$. We set the initial inventory of product $i$ as  $c_i=\lceil \eta \ts d_i\rceil$, where $\eta$ is one last parameter that we vary to control the scarcity of the product inventories. 
Varying the parameters $\kappa\in\{0,0.02,0.10\}$, $P_0\in\{0.1,0.2,0.3\}$ and $\eta\in\{0.5,0.65,0.8\}$ yields $27$ parameter configurations. Note that if $\kappa = 0$, then the arrival probabilities are uniform across types at each time period. For each of the 27 parameter configurations, we use the approach in this section to generate two test problems, one with $T= 200$ and one with $T = 1000$ time periods in the selling horizon. In this case, we obtain a total of 54 test problems in our numerical setup.

\vspace{-1mm}

\section{Choice-Based Deterministic Linear Program}
\label{sec:cdlp}

\vspace{-1mm}

We give the choice-based deterministic linear program that forms the basis of the CDLP policy. This linear program is as a crude approximation to capture the decisions of the optimal policy, constructed under the assumption that the arrivals and choices of the customers take on their expected values; see \cite{RyLi04}. In the linear program, we set $R_j(S)=\sum_{i\in S} r_i\,\phi_{ij}(S)$ to capture the expected revenue that we obtain when we offer the assortment $S$ to a customer of type $j$. We use the decision variables \mbox{$\xvec = (x_j(S) : j \in \Jcal,~S \subseteq \Ncal) \in \mathbb R_+^{m 2^n}$}, where $x_j(S)$ captures  the probability of offering assortment $S$ to a customer of type $j$. Noting that the total expected number of arrivals of customers of type $j$ is $\tau_j = \sum_{t \in \Tcal} \lambda_{jt}$, we consider the linear program 
\begin{align*}
\max_{\xvec \in  \mathbb R_+^{m 2^n}}\;\; & \sum_{j \in \Jcal}\tau_j \sum_{S\subseteq \Ncal} R_j(S)\,x_j(S) \\
\text{s.t.}\quad
& \sum_{j \in \Jcal}\tau_j \sum_{S \subseteq \Ncal} \phi_{ij}(S)\,x_j(S) \;\le\; c_i, &\forall\, i\in\Ncal,\\
& \sum_{S\subseteq \Ncal} x_j(S) \;=\; 1, &\forall\, j\in\Jcal.
\end{align*}

\indent In the objective function, we account for the total expected revenue over the selling horizon. The first constraint ensures that the total expected inventory consumption of product $i$ does not exceed its initial inventory. The second constraint ensures that we offer an assortment to each customer type with probability one. Letting $\xvec^*$ be an optimal solution to the problem above, the CDLP policy offers, at every arrival of a customer of type $j$, assortment $S$ with probability $x_j^*(S)$. The number of decision variables above increases exponentially with the number of products. Thus, we often solve the problem above by using column generation. If the choices are governed by the multinomial logit model, then one can bypass column generation; see \cite{GaRa15}.


\vspace{-1mm}

\section{Details of Our Numerical Results}
\label{sec:exp_details}

\vspace{-1mm}

We give the percentile and variance breakdowns of the total revenues achieved by the policies we test, as well as the running times and detailed per-instance results that complement the highlights of the experimental findings in the main text. In Figure \ref{fig:percentile-profile}, we plot the values of $\Delta P_\alpha$ as a function of the percentile level $\alpha$ for all four baseline policy and de-randomization method pairs. For every pair, the values of $\Delta P_\alpha$ are monotonically decreasing in $\alpha$, indicating that \mbox{de-randomization} yields larger improvements on the lower revenue sample paths. The two data series for the \mbox{uniform-myopic} baseline policy are close with the support method holding a slight edge, indicating that while the global method delivers consistently strong improvements to the original policy, when the original support is already rich, searching beyond it can be detrimental. For the CDLP baseline policy, by contrast, the gap between the support and global de-randomization methods is substantial and grows deeper in the tail. In Table \ref{tab:main-results}, we extend the results in Table~\ref{tab:main-short} to give the performance improvements from \mbox{de-randomization} when we focus on the lower tails of total revenues for the different baseline policy and de-randomization method pairs. The results support the earlier observation that \mbox{de-randomization} yields larger improvements on the lower revenue sample paths, as the values of $\Delta P_\alpha$ are decreasing in $\alpha$.

\begin{figure}[t]
\vspace{-3mm}
\centering
\includegraphics[width=0.49\textwidth]{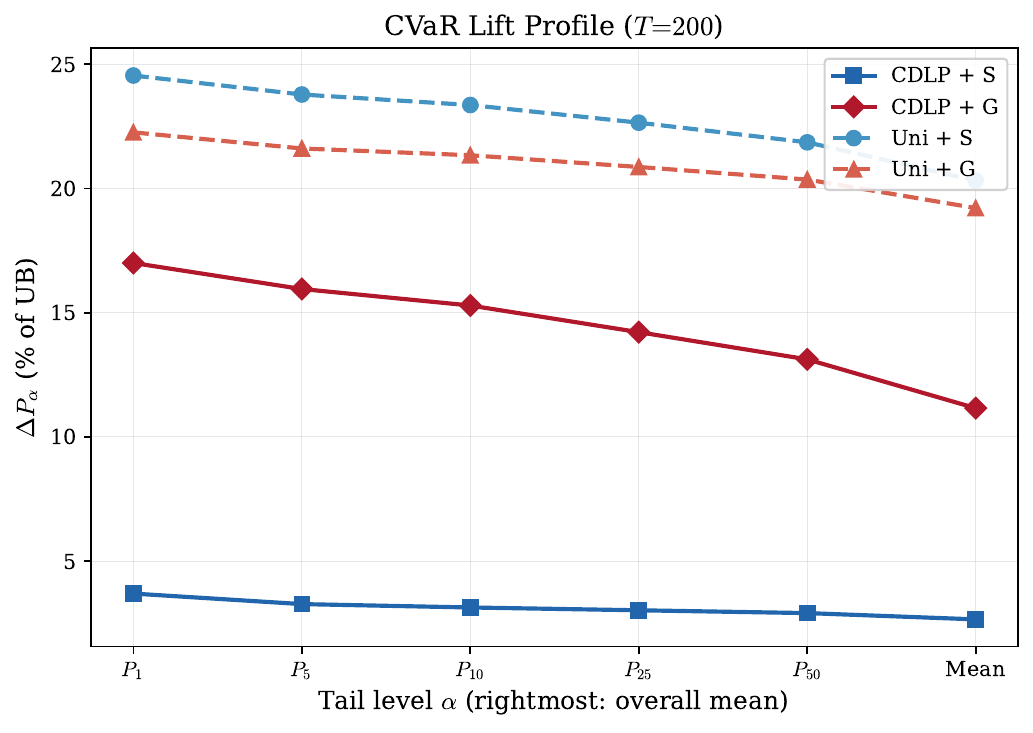}
\includegraphics[width=0.49\textwidth]{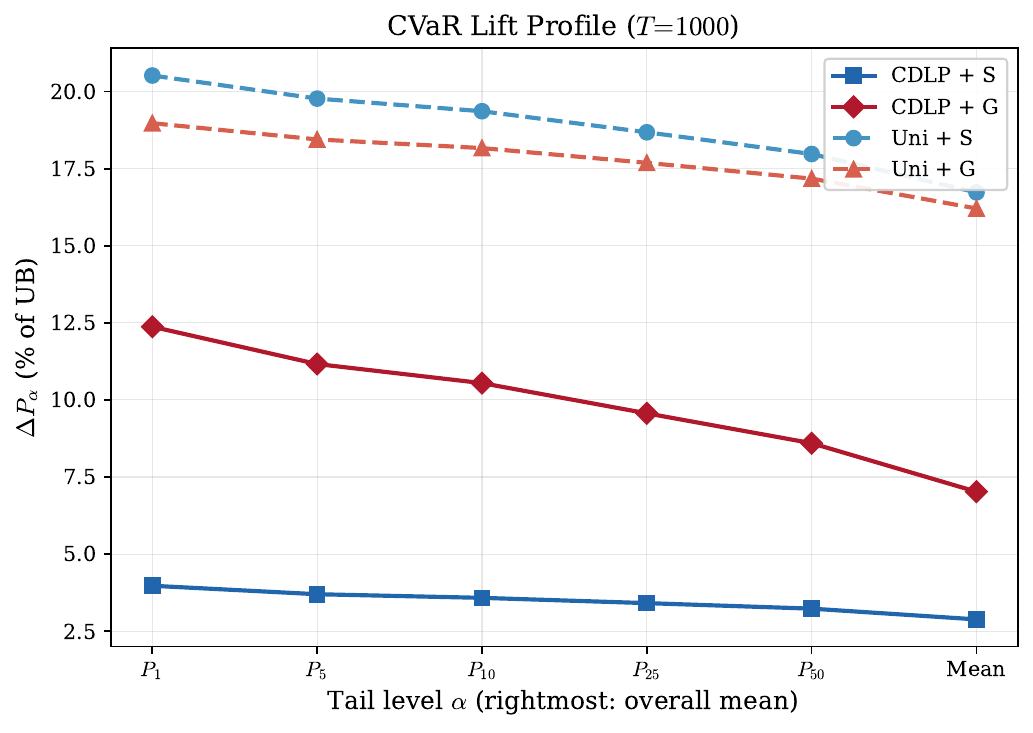}
\caption{\!\!\!\!\!\!Values of $\Delta P_{\alpha}$ as a function of $\alpha$ for all four baseline policy and de-randomization method pairs. 
}
\label{fig:percentile-profile}

\vspace{-4mm}
\end{figure}

\begin{table}[t]
\centering
{\scriptsize
\setlength{\tabcolsep}{4pt}
\begin{tabular}{lrrrrrrrr}
\toprule
 & \multicolumn{4}{c}{$T = 200$} & \multicolumn{4}{c}{$T = 1000$} \\
\cmidrule(lr){2-5} \cmidrule(lr){6-9}
 & \multicolumn{2}{c}{CDLP} & \multicolumn{2}{c}{Uniform-myopic} & \multicolumn{2}{c}{CDLP} & \multicolumn{2}{c}{Uniform-myopic} \\
\cmidrule(lr){2-3} \cmidrule(lr){4-5} \cmidrule(lr){6-7} \cmidrule(lr){8-9}
 & Support & Global & Support & Global & Support & Global & Support & Global \\
\midrule
$\Delta_{\text{\sf agn}}$ & 2.0 & 9.7 & 19.3 & 17.5 & 2.5 & 8.3 & 20.4 & 18.8 \\
$\Delta_{\text{\sf awr}}$ & 2.7 & 11.2 & 20.3 & 19.2 & 2.9 & 7.0 & 16.7 & 16.2 \\
\midrule
$\Delta P_1$ & 3.7 & 17.0 & 24.5 & 22.3 & 4.0 & 12.4 & 20.5 & 19.0 \\
$\Delta P_5$ & 3.3 & 16.0 & 23.8 & 21.6 & 3.7 & 11.2 & 19.8 & 18.4 \\
$\Delta P_{10}$ & 3.1 & 15.3 & 23.4 & 21.3 & 3.6 & 10.5 & 19.4 & 18.2 \\
$\Delta P_{25}$ & 3.0 & 14.2 & 22.6 & 20.9 & 3.4 & 9.6 & 18.7 & 17.7 \\
$\Delta P_{50}$ & 2.9 & 13.1 & 21.9 & 20.4 & 3.2 & 8.6 & 18.0 & 17.2 \\
\bottomrule
\end{tabular}
}
\caption{\!\!\!\!\!\!Improvement from de-randomization for each baseline policy and de-randomization method pair.}
\label{tab:main-results}
\vspace{-8mm}
\end{table}

In Table~\ref{tab:cv-improvement}, we give the coefficient of variation of the total revenue for each baseline policy and de-randomization method pair before and after de-randomization, demonstrating that \mbox{de-randomization} also reduces the total revenue variability. We compute the coefficient of variation of the total revenue of a policy as $100\%\times\sigma/\mu$, where $\mu$ and $\sigma$ are, respectively, the mean and standard deviation of the total revenue across the 1000 Monte Carlo replications. Coefficient of variation is non-zero even for a de-randomized policy because customer arrivals and purchase outcomes are random; de-randomization eliminates only the  randomness in the offered assortments. For each test problem, the coefficient of variation reflects the variability in the total revenue of a policy, considering the randomness in the assortments offered by the policy, types of the arriving customers and purchase decisions. The coefficients of variation in Table~\ref{tab:cv-improvement} are averages of these per-instance coefficients of variation across the 27 test problems with different parameter configurations. We use $\text{CV}_\text{\sf rand}$  and $\text{CV}_\text{\sf det}$ to capture the coefficient of variation of a baseline policy and de-randomization method pair, respectively, before and after de-randomization. The largest relative coefficient of variation reduction of 92.2\% comes from the \mbox{uniform-myopic} baseline policy with the support de-randomization method, but even the most conservative pair, the CDLP baseline policy with the support \mbox{de-randomization} method, significantly reduces the coefficient of variation. The effect is especially pronounced for larger number of time periods in the selling horizon with $T=1000$, where the best combination leaves only a  small residual coefficient of variation.

\begin{table}[t]

\centering
{\scriptsize
\setlength{\tabcolsep}{4pt}
\begin{tabular}{lrrrrrrrr}
\toprule
 & \multicolumn{4}{c}{$T = 200$} & \multicolumn{4}{c}{$T = 1000$} \\
\cmidrule(lr){2-5} \cmidrule(lr){6-9}
 & \multicolumn{2}{c}{CDLP} & \multicolumn{2}{c}{Uniform-myopic} & \multicolumn{2}{c}{CDLP} & \multicolumn{2}{c}{Uniform-myopic} \\
\cmidrule(lr){2-3} \cmidrule(lr){4-5} \cmidrule(lr){6-7} \cmidrule(lr){8-9}
 & Support & Global & Support & Global & Support & Global & Support & Global \\
\midrule
CV$_{\text{\sf rand}}$ (\%) & \multicolumn{2}{c}{5.09} & \multicolumn{2}{c}{5.06} & \multicolumn{2}{c}{2.32} & \multicolumn{2}{c}{2.04} \\
CV$_{\text{\sf det}}$ (\%) & 4.58 & 2.02 & 2.03 & 2.58 & 1.80 & 0.22 & 0.17 & 0.50 \\
Relative reduction (\%) & 9.7 & 58.2 & 62.0 & 50.9 & 21.7 & 89.3 & 92.2 & 77.0 \\
\bottomrule
\end{tabular}
}
\caption{Coefficient of variation of  total revenue for each baseline policy and de-randomization method pair.}
\label{tab:cv-improvement}
\vspace{-0mm}
\end{table}

If the baseline policy does not randomize over too many assortments, then we expect the benefits from de-randomization to be more modest. To understand the role of randomness in the initial policy on the benefits from de-randomization, we summarize the extent of randomization in policy $\mu$ by the arrival-weighted entropy $H(\mu) = \frac{1}{T}\sum_{t \in \Tcal} \sum_{j \in \Jcal} \lambda_{jt}\, H_{jt}(\mu)$, where $H_{jt}(\mu)$ is the Shannon entropy of the assortment distribution used at customer type and time period pair $(j,t)$ given by $H_{jt}(\mu)=-\sum_{S\in\Omega^{\mu}_{jt}} \mathbb{P}\{\Srm^{\mu}_{jt}=S\}\log_2 \mathbb{P}\{\Srm^{\mu}_{jt}=S\}$. In our test problems with $T = 1000$, on average, the CDLP baseline policy has mean entropy of $0.86$ bits and the uniform-myopic baseline policy has mean entropy of $3.92$ bits, reflecting the much narrower support sizes of the CDLP baseline policy relative to the uniform-myopic one. In Figure~\ref{fig:entropy-scatter}, we plot the arrival-weighted entropy of the randomized baseline policy against the performance improvement from de-randomization given by $\Delta_{\text{\sf awr}}$ at $T=1000$, separately for the support and global de-randomization methods. The CDLP data points cluster at low entropy with smaller gains, while the uniform-myopic data points occupy a higher-entropy region with larger gains. Under the support de-randomization method, the CDLP data points show essentially no correlation between entropy and improvement; this is expected because the support of the CDLP baseline policy contains only a handful of assortments, leaving the support method with limited room for improvement regardless of the entropy level. Under the global de-randomization method, however, a clear positive relation emerges. In particular, once the search is expanded beyond the narrow original support, higher initial randomness in the policy translates directly into greater room for improvement. We emphasize that the correlations that we read from Figure~\ref{fig:entropy-scatter} are not statistically significant at the size of our experiments with 27 parameter configurations; they are included only to illustrate qualitative trends.

\begin{figure}[t]
\centering
\includegraphics[width=0.49\textwidth]{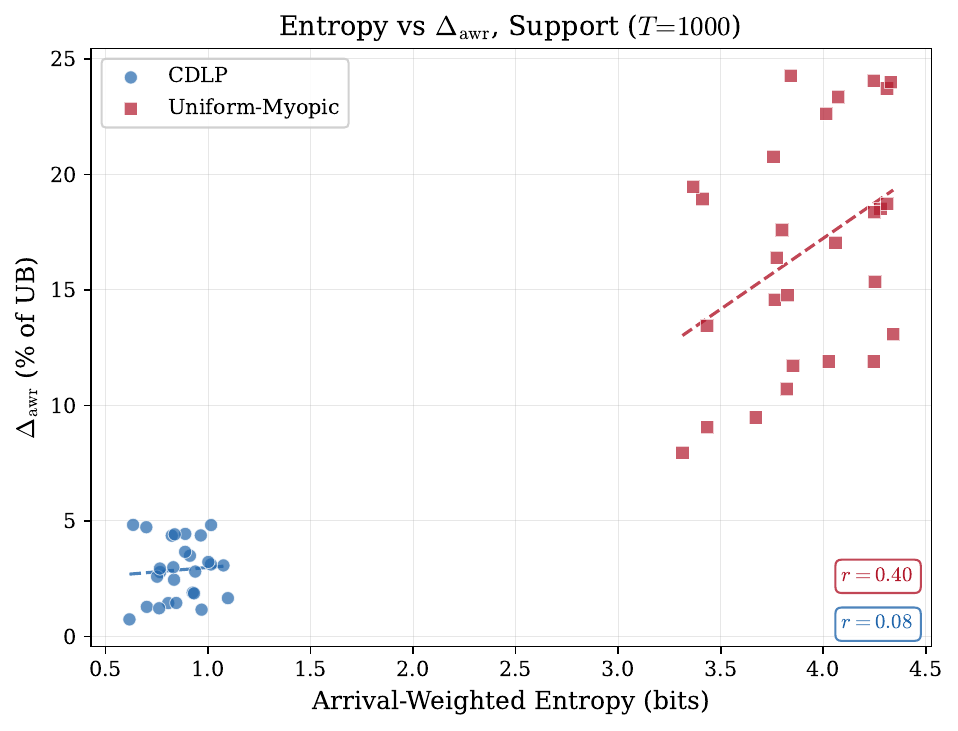}
\includegraphics[width=0.49\textwidth]{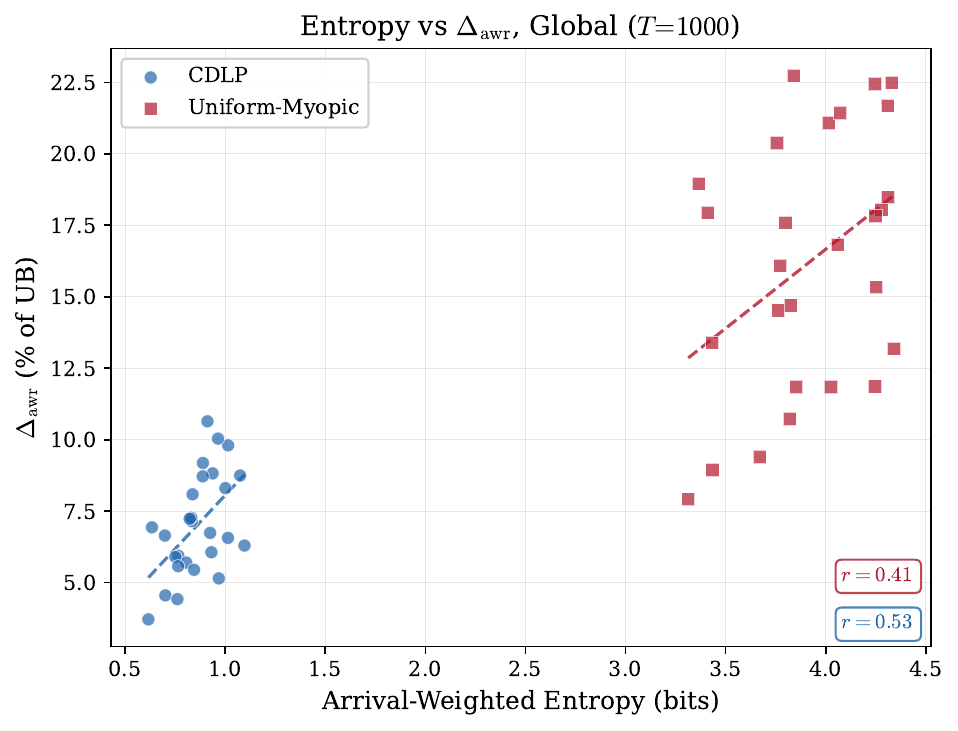}
\caption{Arrival-weighted entropy in bits versus $\Delta_{\text{\sf awr}}$ at $T=1000$, grouped by de-randomization method. Each panel shows both baselines with separate regression lines and Pearson correlations.}
\label{fig:entropy-scatter}
\vspace{-2mm}
\end{figure}

In Table~\ref{tab:runtime}, we give the de-randomization running times in seconds, averaged over all instances for the two possible values of the number of time periods in the selling horizon. Empirically, running time grows more slowly than the worst-case $T^3$ scaling of the oracle bound  that we give in Section~\ref{sec:derand}.~The global de-randomization method is slightly slower because it optimizes over a larger candidate family at each coordinate, but each time period requires only a single oracle call as discussed just after (\ref{eqn:rev_dec}) in the main text  and previous oracle computations can often be reused, so the overhead for the global de-randomization method remains modest.

\raggedbottom
\begin{table}[t]
\centering
{\small
\begin{tabular}{crrrr}
\toprule
 & \multicolumn{2}{c}{Support} & \multicolumn{2}{c}{Global} \\
\cmidrule(lr){2-3} \cmidrule(lr){4-5}
$T$ & Mean & Max & Mean & Max \\
\midrule
~200 & 3.4 & 4.9 & 6.4 & 8.3 \\
1000 & 160.8 & 211.4 & 175.8 & 226.5 \\
\bottomrule
\end{tabular}
}
\caption{Mean de-randomization running time per instance in seconds.}
\label{tab:runtime}
\vspace{-8mm}
\end{table}

Finally, Tables~\ref{tab:raw-cdlp-200}-\ref{tab:raw-unif-1000} collect the instance-level results for the $27$ parameter configurations at both selling horizon lengths. The format of all of these tables is the same. In the first column, we give the parameter tuple $(\kappa,P_0,\eta)$, where, as in Appendix \ref{sec:gen}, $\kappa$ controls the rate of decline of the arrival probabilities as we move away from the peak arrival time period, $P_0$ is the no-purchase probability when the full consideration set is offered and $\eta$ is the capacity scaling parameter for the product inventories. In the second column, we give the arrival-weighted entropy for the initial baseline policy. In the third column, we give the average initial inventory per product, given by $\overline c = \sum_{i \in \Ncal} c_i / n$. In the rest of the tables, there are two blocks. In the first block, we focus on the support de-randomization method. We give the values of $\Delta_{\text{\sf awr}}$, as well as $\Delta P_1$, $\Delta P_5$ and $\Delta P_{25}$, characterizing the performance improvement from de-randomization on average and when we focus on the sample paths with smaller percentiles of total revenues. In the second block, we give the same statistics for the global de-randomization method. The last row reports unweighted averages across the $27$ parameter configurations. In Tables \ref{tab:raw-cdlp-200} and \ref{tab:raw-unif-200}, we focus on the test problems with $T = 200$, where each of these two tables focuses on the CDLP and uniform-myopic baseline policies. In Tables \ref{tab:raw-cdlp-1000} and \ref{tab:raw-unif-1000}, we focus on the test problems with $T = 1000$, where, once again, each of these two tables focuses on the CDLP and uniform-myopic baseline policies.

\renewcommand{\arraystretch}{0.92}

\begin{table}[t]
\vspace{-4mm}
\centering
\scriptsize
\begin{tabular}{cccrrrrrrrr}
\toprule
$(\kappa,P_0,\eta)$ & $H$ & $\bar{c}$ & \multicolumn{4}{c}{Support} & \multicolumn{4}{c}{Global} \\
\cmidrule(lr){4-7} \cmidrule(lr){8-11}
 & & & $\Delta_{\text{\sf awr}}$ & $\Delta P_1$ & $\Delta P_5$ & $\Delta P_{25}$ & $\Delta_{\text{\sf awr}}$ & $\Delta P_1$ & $\Delta P_5$ & $\Delta P_{25}$ \\
\midrule
(0.00,\,0.1,\,0.50) & 0.64 & 1.7 & 4.3 & 5.1 & 5.8 & 5.6 & 14.7 & 26.9 & 25.2 & 21.0 \\
(0.00,\,0.1,\,0.65) & 0.68 & 2.3 & 3.1 & 5.2 & 4.7 & 3.9 & 12.5 & 21.8 & 19.6 & 17.2 \\
(0.00,\,0.1,\,0.80) & 0.80 & 2.8 & 0.9 & 0.7 & 0.9 & 0.9 & 7.5 & 10.8 & 10.3 & 9.4 \\
(0.00,\,0.2,\,0.50) & 0.79 & 1.7 & 3.7 & 3.5 & 4.1 & 4.0 & 17.1 & 26.6 & 24.9 & 22.1 \\
(0.00,\,0.2,\,0.65) & 0.83 & 2.0 & 2.8 & 3.9 & 3.2 & 3.2 & 12.2 & 18.5 & 17.2 & 15.4 \\
(0.00,\,0.2,\,0.80) & 0.69 & 2.4 & 1.8 & 2.5 & 2.2 & 1.7 & 6.4 & 8.6 & 8.4 & 7.7 \\
(0.00,\,0.3,\,0.50) & 0.94 & 1.6 & 3.0 & 4.4 & 3.4 & 3.3 & 17.9 & 25.5 & 24.4 & 22.1 \\
(0.00,\,0.3,\,0.65) & 1.01 & 1.8 & 1.7 & 3.3 & 1.7 & 1.8 & 12.8 & 16.9 & 16.1 & 15.0 \\
(0.00,\,0.3,\,0.80) & 0.74 & 2.3 & 1.8 & 2.9 & 2.4 & 2.2 & 3.9 & 5.4 & 5.6 & 4.7 \\
(0.02,\,0.1,\,0.50) & 0.72 & 1.8 & 4.7 & 8.4 & 6.5 & 5.6 & 14.1 & 26.6 & 24.1 & 20.2 \\
(0.02,\,0.1,\,0.65) & 0.86 & 2.3 & 3.7 & 5.0 & 4.4 & 4.0 & 11.4 & 18.4 & 16.9 & 14.8 \\
(0.02,\,0.1,\,0.80) & 0.88 & 2.7 & 0.9 & 0.6 & 0.8 & 1.0 & 6.2 & 10.3 & 8.8 & 7.7 \\
(0.02,\,0.2,\,0.50) & 0.90 & 1.7 & 3.9 & 3.4 & 4.2 & 4.4 & 17.0 & 27.7 & 25.0 & 22.2 \\
(0.02,\,0.2,\,0.65) & 0.82 & 2.0 & 2.8 & 4.8 & 4.2 & 3.5 & 13.0 & 19.7 & 18.3 & 16.6 \\
(0.02,\,0.2,\,0.80) & 0.71 & 2.5 & 1.1 & 3.1 & 2.1 & 1.2 & 5.9 & 6.3 & 6.7 & 6.7 \\
(0.02,\,0.3,\,0.50) & 1.01 & 1.6 & 3.3 & 3.0 & 3.0 & 3.3 & 17.1 & 23.4 & 22.7 & 20.9 \\
(0.02,\,0.3,\,0.65) & 1.02 & 1.9 & 2.3 & 1.8 & 1.9 & 2.3 & 11.5 & 16.4 & 14.9 & 13.6 \\
(0.02,\,0.3,\,0.80) & 0.68 & 2.4 & 0.7 & 1.6 & 1.1 & 0.6 & 1.7 & 1.4 & 1.9 & 1.9 \\
(0.10,\,0.1,\,0.50) & 0.83 & 1.9 & 6.0 & 7.4 & 6.3 & 6.5 & 16.0 & 25.0 & 23.9 & 21.5 \\
(0.10,\,0.1,\,0.65) & 0.92 & 2.4 & 3.0 & 4.2 & 4.1 & 3.5 & 11.4 & 18.4 & 17.0 & 15.0 \\
(0.10,\,0.1,\,0.80) & 0.60 & 2.7 & 1.4 & 2.3 & 1.7 & 1.7 & 6.4 & 9.7 & 9.1 & 8.2 \\
(0.10,\,0.2,\,0.50) & 0.86 & 1.6 & 4.5 & 5.4 & 5.4 & 5.1 & 15.9 & 25.5 & 24.0 & 20.9 \\
(0.10,\,0.2,\,0.65) & 0.99 & 2.1 & 2.7 & 4.0 & 3.0 & 3.0 & 11.8 & 17.5 & 16.4 & 14.4 \\
(0.10,\,0.2,\,0.80) & 0.94 & 2.6 & 1.0 & 0.8 & 0.9 & 1.2 & 4.1 & 6.0 & 5.5 & 5.0 \\
(0.10,\,0.3,\,0.50) & 0.93 & 1.5 & 3.1 & 4.3 & 3.5 & 3.3 & 16.5 & 23.7 & 22.6 & 20.3 \\
(0.10,\,0.3,\,0.65) & 1.01 & 1.9 & 2.3 & 5.1 & 4.2 & 3.0 & 12.3 & 17.3 & 16.1 & 14.7 \\
(0.10,\,0.3,\,0.80) & 0.96 & 2.3 & 1.1 & 3.2 & 2.5 & 1.7 & 3.9 & 4.9 & 4.9 & 4.7 \\
\midrule
Mean & & & 2.7 & 3.7 & 3.3 & 3.0 & 11.2 & 17.0 & 16.0 & 14.2 \\
\bottomrule
\end{tabular}
\caption{Detailed results for the CDLP baseline policy and $T = 200$.}\label{tab:raw-cdlp-200}
\vspace{-5mm}
\end{table}

\begin{table}[!ht]
\vspace{-4mm}
\centering
\scriptsize
\begin{tabular}{cccrrrrrrrr}
\toprule
$(\kappa,P_0,\eta)$ & $H$ & $\bar{c}$ & \multicolumn{4}{c}{Support} & \multicolumn{4}{c}{Global} \\
\cmidrule(lr){4-7} \cmidrule(lr){8-11}
 & & & $\Delta_{\text{\sf awr}}$ & $\Delta P_1$ & $\Delta P_5$ & $\Delta P_{25}$ & $\Delta_{\text{\sf awr}}$ & $\Delta P_1$ & $\Delta P_5$ & $\Delta P_{25}$ \\
\midrule
(0.00,\,0.1,\,0.50) & 3.31 & 1.7 & 12.3 & 20.3 & 18.6 & 16.5 & 12.0 & 18.0 & 17.2 & 15.8 \\
(0.00,\,0.1,\,0.65) & 3.76 & 2.3 & 17.8 & 23.8 & 22.9 & 21.7 & 16.5 & 19.8 & 19.6 & 18.8 \\
(0.00,\,0.1,\,0.80) & 3.76 & 2.8 & 23.6 & 28.0 & 27.2 & 25.8 & 20.3 & 22.6 & 21.7 & 21.2 \\
(0.00,\,0.2,\,0.50) & 3.82 & 1.7 & 16.1 & 20.7 & 20.2 & 18.8 & 16.3 & 19.5 & 19.0 & 18.6 \\
(0.00,\,0.2,\,0.65) & 3.80 & 2.0 & 21.5 & 25.6 & 25.1 & 23.9 & 20.4 & 23.3 & 23.1 & 22.3 \\
(0.00,\,0.2,\,0.80) & 3.84 & 2.4 & 25.3 & 27.5 & 27.1 & 26.6 & 23.3 & 24.3 & 24.5 & 24.1 \\
(0.00,\,0.3,\,0.50) & 4.34 & 1.6 & 21.6 & 26.6 & 26.1 & 24.5 & 20.7 & 24.6 & 23.6 & 22.7 \\
(0.00,\,0.3,\,0.65) & 4.28 & 1.8 & 21.6 & 24.3 & 23.7 & 23.3 & 20.9 & 24.4 & 22.9 & 22.2 \\
(0.00,\,0.3,\,0.80) & 4.25 & 2.3 & 23.4 & 25.5 & 24.4 & 23.9 & 22.2 & 23.1 & 22.2 & 22.6 \\
(0.02,\,0.1,\,0.50) & 3.44 & 1.8 & 13.5 & 22.2 & 19.7 & 17.3 & 13.4 & 20.1 & 19.1 & 17.0 \\
(0.02,\,0.1,\,0.65) & 3.43 & 2.3 & 18.8 & 24.1 & 22.9 & 21.4 & 17.7 & 21.8 & 20.8 & 19.8 \\
(0.02,\,0.1,\,0.80) & 3.41 & 2.7 & 21.3 & 24.4 & 23.8 & 22.9 & 18.7 & 19.6 & 19.7 & 19.6 \\
(0.02,\,0.2,\,0.50) & 4.03 & 1.7 & 17.9 & 24.0 & 23.2 & 21.4 & 17.4 & 23.6 & 22.1 & 20.3 \\
(0.02,\,0.2,\,0.65) & 3.78 & 2.0 & 20.1 & 23.8 & 23.2 & 22.5 & 18.4 & 21.1 & 20.3 & 19.8 \\
(0.02,\,0.2,\,0.80) & 4.01 & 2.5 & 24.2 & 25.6 & 25.6 & 25.3 & 22.1 & 22.5 & 22.5 & 22.5 \\
(0.02,\,0.3,\,0.50) & 4.25 & 1.6 & 18.9 & 21.2 & 21.9 & 21.2 & 18.5 & 21.6 & 20.9 & 20.3 \\
(0.02,\,0.3,\,0.65) & 4.32 & 1.9 & 21.5 & 24.2 & 23.8 & 23.0 & 21.0 & 23.5 & 22.8 & 22.1 \\
(0.02,\,0.3,\,0.80) & 4.31 & 2.4 & 22.0 & 23.4 & 22.9 & 22.5 & 21.5 & 23.9 & 22.7 & 22.0 \\
(0.10,\,0.1,\,0.50) & 3.67 & 1.9 & 13.9 & 21.5 & 19.9 & 17.9 & 13.6 & 18.9 & 18.1 & 16.7 \\
(0.10,\,0.1,\,0.65) & 3.82 & 2.4 & 19.6 & 26.0 & 25.3 & 23.3 & 18.4 & 22.0 & 21.5 & 20.5 \\
(0.10,\,0.1,\,0.80) & 3.37 & 2.7 & 20.7 & 24.8 & 23.8 & 22.5 & 18.8 & 20.2 & 19.8 & 19.4 \\
(0.10,\,0.2,\,0.50) & 3.85 & 1.6 & 18.8 & 25.4 & 24.1 & 22.1 & 18.2 & 22.9 & 21.6 & 20.6 \\
(0.10,\,0.2,\,0.65) & 4.06 & 2.1 & 21.8 & 27.3 & 26.1 & 24.4 & 20.2 & 23.7 & 23.0 & 21.9 \\
(0.10,\,0.2,\,0.80) & 4.07 & 2.6 & 23.2 & 23.3 & 23.8 & 24.0 & 21.5 & 23.4 & 22.6 & 22.0 \\
(0.10,\,0.3,\,0.50) & 4.25 & 1.5 & 23.0 & 28.6 & 27.5 & 26.0 & 22.3 & 26.1 & 25.5 & 24.5 \\
(0.10,\,0.3,\,0.65) & 4.25 & 1.9 & 22.8 & 24.8 & 24.8 & 24.4 & 22.1 & 24.1 & 23.8 & 23.3 \\
(0.10,\,0.3,\,0.80) & 4.33 & 2.3 & 23.5 & 25.8 & 24.7 & 24.5 & 22.5 & 22.3 & 22.9 & 22.9 \\
\midrule
Mean & & & 20.3 & 24.5 & 23.8 & 22.6 & 19.2 & 22.3 & 21.6 & 20.9 \\
\bottomrule
\end{tabular}
\caption{Detailed results for the uniform-myopic baseline policy and $T = 200$.}\label{tab:raw-unif-200}
\vspace{-5mm}
\end{table}

\begin{table}[!ht]
\centering
\scriptsize
\begin{tabular}{ccrrrrrrrrr}
\toprule
$(\kappa,P_0,\eta)$ & $H$ & $\bar{c}~~$ & \multicolumn{4}{c}{Support} & \multicolumn{4}{c}{Global} \\
\cmidrule(lr){4-7} \cmidrule(lr){8-11}
 & & & $\Delta_{\text{\sf awr}}$ & $\Delta P_1$ & $\Delta P_5$ & $\Delta P_{25}$ & $\Delta_{\text{\sf awr}}$ & $\Delta P_1$ & $\Delta P_5$ & $\Delta P_{25}$ \\
\midrule
(0.00,\,0.1,\,0.50) & 0.63 & 7.8 & 4.8 & 7.2 & 6.7 & 5.9 & 6.9 & 13.9 & 12.3 & 10.1 \\
(0.00,\,0.1,\,0.65) & 0.77 & 10.7 & 2.8 & 4.3 & 3.9 & 3.3 & 5.9 & 11.3 & 9.9 & 8.3 \\
(0.00,\,0.1,\,0.80) & 0.70 & 13.0 & 1.3 & 1.9 & 1.6 & 1.5 & 4.6 & 8.6 & 7.8 & 6.5 \\
(0.00,\,0.2,\,0.50) & 0.89 & 7.1 & 4.4 & 5.7 & 5.5 & 5.2 & 9.2 & 15.6 & 14.5 & 12.4 \\
(0.00,\,0.2,\,0.65) & 0.83 & 9.0 & 2.5 & 2.2 & 2.6 & 2.9 & 7.1 & 13.6 & 12.0 & 10.1 \\
(0.00,\,0.2,\,0.80) & 0.81 & 10.9 & 1.4 & 1.9 & 1.9 & 1.8 & 5.7 & 10.1 & 9.1 & 7.9 \\
(0.00,\,0.3,\,0.50) & 0.91 & 6.3 & 3.5 & 4.1 & 4.3 & 4.0 & 10.6 & 18.7 & 16.9 & 14.4 \\
(0.00,\,0.3,\,0.65) & 0.94 & 8.1 & 2.8 & 5.2 & 4.2 & 3.5 & 8.8 & 15.2 & 13.7 & 11.8 \\
(0.00,\,0.3,\,0.80) & 0.93 & 10.3 & 1.9 & 2.3 & 2.1 & 2.0 & 6.7 & 10.3 & 9.4 & 8.5 \\
(0.02,\,0.1,\,0.50) & 0.70 & 7.9 & 4.7 & 7.9 & 6.8 & 5.9 & 6.6 & 13.4 & 11.7 & 9.6 \\
(0.02,\,0.1,\,0.65) & 0.75 & 10.3 & 2.6 & 2.9 & 2.9 & 2.9 & 5.9 & 11.1 & 10.0 & 8.4 \\
(0.02,\,0.1,\,0.80) & 0.76 & 12.5 & 1.2 & 1.1 & 1.3 & 1.4 & 4.4 & 7.7 & 6.8 & 5.9 \\
(0.02,\,0.2,\,0.50) & 0.89 & 7.3 & 3.7 & 4.5 & 4.3 & 4.1 & 8.7 & 16.0 & 14.6 & 12.2 \\
(0.02,\,0.2,\,0.65) & 0.83 & 8.9 & 3.0 & 4.6 & 4.0 & 3.6 & 7.3 & 13.0 & 11.6 & 10.0 \\
(0.02,\,0.2,\,0.80) & 0.84 & 11.5 & 1.4 & 1.8 & 1.7 & 1.7 & 5.4 & 9.1 & 8.3 & 7.2 \\
(0.02,\,0.3,\,0.50) & 1.01 & 6.4 & 4.8 & 6.1 & 6.3 & 5.8 & 9.8 & 17.7 & 15.9 & 13.5 \\
(0.02,\,0.3,\,0.65) & 1.07 & 8.2 & 3.1 & 4.0 & 3.5 & 3.5 & 8.7 & 13.8 & 13.0 & 11.5 \\
(0.02,\,0.3,\,0.80) & 1.10 & 10.2 & 1.7 & 2.0 & 1.8 & 1.8 & 6.3 & 9.9 & 9.0 & 8.0 \\
(0.10,\,0.1,\,0.50) & 0.82 & 8.2 & 4.4 & 6.5 & 5.8 & 5.3 & 7.2 & 13.3 & 12.0 & 10.1 \\
(0.10,\,0.1,\,0.65) & 0.77 & 10.8 & 2.9 & 4.1 & 3.9 & 3.6 & 5.6 & 9.8 & 8.7 & 7.5 \\
(0.10,\,0.1,\,0.80) & 0.62 & 12.4 & 0.7 & 1.6 & 1.2 & 0.9 & 3.7 & 6.7 & 5.8 & 5.1 \\
(0.10,\,0.2,\,0.50) & 0.84 & 6.9 & 4.4 & 6.1 & 5.6 & 5.2 & 8.1 & 15.0 & 13.5 & 11.5 \\
(0.10,\,0.2,\,0.65) & 1.01 & 9.2 & 3.1 & 4.8 & 4.2 & 3.8 & 6.6 & 11.7 & 10.5 & 9.1 \\
(0.10,\,0.2,\,0.80) & 0.97 & 11.7 & 1.2 & 0.7 & 1.2 & 1.4 & 5.1 & 8.0 & 7.3 & 6.5 \\
(0.10,\,0.3,\,0.50) & 0.96 & 6.2 & 4.4 & 7.0 & 6.2 & 5.3 & 10.0 & 17.1 & 15.6 & 13.5 \\
(0.10,\,0.3,\,0.65) & 1.00 & 8.4 & 3.2 & 4.3 & 3.8 & 3.7 & 8.3 & 13.7 & 12.6 & 10.9 \\
(0.10,\,0.3,\,0.80) & 0.93 & 10.1 & 1.9 & 2.3 & 2.4 & 2.1 & 6.1 & 10.1 & 8.9 & 7.7 \\
\midrule
Mean & & & 2.9 & 4.0 & 3.7 & 3.4 & 7.0 & 12.4 & 11.2 & 9.6 \\
\bottomrule
\end{tabular}
\caption{Detailed results for the CDLP baseline policy and $T = 1000$.}\label{tab:raw-cdlp-1000}
\end{table}

\begin{table}[!ht]
\vspace{-4mm}
\centering
\scriptsize
\begin{tabular}{ccrrrrrrrrr}
\toprule
$(\kappa,P_0,\eta)$ & $H$ & $\bar{c}~~$ & \multicolumn{4}{c}{Support} & \multicolumn{4}{c}{Global} \\
\cmidrule(lr){4-7} \cmidrule(lr){8-11}
 & & & $\Delta_{\text{\sf awr}}$ & $\Delta P_1$ & $\Delta P_5$ & $\Delta P_{25}$ & $\Delta_{\text{\sf awr}}$ & $\Delta P_1$ & $\Delta P_5$ & $\Delta P_{25}$ \\
\midrule
(0.00,\,0.1,\,0.50) & 3.31 & 7.8 & 7.9 & 12.0 & 11.2 & 10.0 & 7.9 & 12.0 & 11.1 & 9.9 \\
(0.00,\,0.1,\,0.65) & 3.76 & 10.7 & 14.6 & 18.3 & 17.5 & 16.4 & 14.5 & 18.1 & 17.5 & 16.4 \\
(0.00,\,0.1,\,0.80) & 3.76 & 13.0 & 20.8 & 24.6 & 23.8 & 22.6 & 20.4 & 22.5 & 22.3 & 21.6 \\
(0.00,\,0.2,\,0.50) & 3.82 & 7.1 & 10.7 & 14.4 & 13.8 & 12.7 & 10.7 & 14.4 & 13.7 & 12.6 \\
(0.00,\,0.2,\,0.65) & 3.80 & 9.0 & 17.6 & 22.0 & 21.1 & 19.8 & 17.6 & 20.5 & 20.3 & 19.5 \\
(0.00,\,0.2,\,0.80) & 3.84 & 10.9 & 24.3 & 27.6 & 27.2 & 26.3 & 22.7 & 24.2 & 23.9 & 23.6 \\
(0.00,\,0.3,\,0.50) & 4.34 & 6.3 & 13.1 & 17.7 & 16.8 & 15.4 & 13.2 & 17.8 & 16.8 & 15.4 \\
(0.00,\,0.3,\,0.65) & 4.28 & 8.1 & 18.5 & 22.9 & 22.1 & 20.8 & 18.0 & 21.3 & 20.6 & 19.7 \\
(0.00,\,0.3,\,0.80) & 4.25 & 10.3 & 24.1 & 26.8 & 26.2 & 25.6 & 22.4 & 23.1 & 23.0 & 23.0 \\
(0.02,\,0.1,\,0.50) & 3.43 & 7.9 & 9.1 & 13.6 & 12.5 & 11.1 & 8.9 & 13.6 & 12.3 & 10.9 \\
(0.02,\,0.1,\,0.65) & 3.43 & 10.3 & 13.4 & 17.8 & 16.8 & 15.6 & 13.4 & 17.3 & 16.6 & 15.4 \\
(0.02,\,0.1,\,0.80) & 3.41 & 12.5 & 18.9 & 22.0 & 21.5 & 20.7 & 17.9 & 19.4 & 19.0 & 18.7 \\
(0.02,\,0.2,\,0.50) & 4.03 & 7.3 & 11.9 & 15.9 & 15.0 & 13.8 & 11.8 & 15.4 & 14.7 & 13.7 \\
(0.02,\,0.2,\,0.65) & 3.77 & 8.9 & 16.4 & 20.3 & 19.8 & 18.5 & 16.1 & 19.1 & 18.3 & 17.6 \\
(0.02,\,0.2,\,0.80) & 4.01 & 11.5 & 22.6 & 25.6 & 25.3 & 24.5 & 21.1 & 22.1 & 22.0 & 21.7 \\
(0.02,\,0.3,\,0.50) & 4.25 & 6.4 & 11.9 & 16.1 & 15.3 & 14.0 & 11.9 & 15.9 & 15.1 & 14.0 \\
(0.02,\,0.3,\,0.65) & 4.31 & 8.2 & 18.7 & 23.1 & 22.0 & 20.9 & 18.5 & 21.5 & 20.9 & 20.2 \\
(0.02,\,0.3,\,0.80) & 4.31 & 10.2 & 23.7 & 26.2 & 25.9 & 25.2 & 21.7 & 22.4 & 22.5 & 22.2 \\
(0.10,\,0.1,\,0.50) & 3.67 & 8.2 & 9.5 & 13.4 & 12.6 & 11.4 & 9.4 & 13.1 & 12.4 & 11.3 \\
(0.10,\,0.1,\,0.65) & 3.83 & 10.8 & 14.8 & 18.6 & 17.6 & 16.5 & 14.7 & 17.9 & 17.1 & 16.3 \\
(0.10,\,0.1,\,0.80) & 3.37 & 12.4 & 19.5 & 22.6 & 21.9 & 21.0 & 19.0 & 20.9 & 20.5 & 19.9 \\
(0.10,\,0.2,\,0.50) & 3.85 & 6.9 & 11.7 & 15.8 & 15.1 & 13.8 & 11.8 & 16.1 & 15.1 & 13.9 \\
(0.10,\,0.2,\,0.65) & 4.06 & 9.2 & 17.0 & 21.1 & 20.3 & 19.1 & 16.8 & 19.1 & 18.8 & 18.2 \\
(0.10,\,0.2,\,0.80) & 4.07 & 11.7 & 23.3 & 26.1 & 25.7 & 24.9 & 21.4 & 21.9 & 21.9 & 21.9 \\
(0.10,\,0.3,\,0.50) & 4.25 & 6.2 & 15.3 & 20.3 & 19.1 & 17.6 & 15.3 & 19.6 & 18.8 & 17.6 \\
(0.10,\,0.3,\,0.65) & 4.25 & 8.4 & 18.4 & 22.5 & 21.4 & 20.5 & 17.8 & 19.5 & 19.5 & 19.1 \\
(0.10,\,0.3,\,0.80) & 4.33 & 10.1 & 24.0 & 26.9 & 26.4 & 25.5 & 22.5 & 23.8 & 23.5 & 23.2 \\
\midrule
Mean & & & 16.7 & 20.5 & 19.8 & 18.7 & 16.2 & 19.0 & 18.4 & 17.7 \\
\bottomrule
\end{tabular}
\caption{Detailed results for the uniform-myopic baseline policy and $T = 1000$.}\label{tab:raw-unif-1000}
\end{table}

\renewcommand{\arraystretch}{1.0}

\end{APPENDICES}

\end{document}